\documentclass[11pt]{amsart}

\usepackage{amsmath,amssymb,amsbsy,amsfonts,amsthm,latexsym,
         amsopn,amstext,amsxtra,euscript,amscd,amsthm, mathtools,dsfont}
         \usepackage{fullpage,amssymb,amsfonts,amsmath}
\usepackage{enumerate}
\usepackage[shortlabels]{enumitem}
\usepackage{comment}
\usepackage{algorithm, algorithmic}
\usepackage{graphicx, subfigure} 
\usepackage{hyperref} 
\usepackage[capitalise]{cleveref}
\usepackage{url}
\usepackage{array}
\usepackage{todonotes}
\usepackage{bbm}
\usepackage{mathrsfs} 
\usepackage{diagbox}
\usepackage[section]{placeins}

\newtheorem{theorem}{Theorem}[section]

\newtheorem{lemma}[theorem]{Lemma}
\newtheorem{conjecture}[theorem]{Conjecture}
\newtheorem{question}[theorem]{Question}

\newtheorem*{theorem*}{Theorem}

\theoremstyle{remark}
\newtheorem*{remark}{Remark}

\renewcommand{\epsilon}{\varepsilon}

\numberwithin{equation}{section}
\numberwithin{figure}{section}
\numberwithin{table}{section}

\makeatletter
\let\c@table\c@figure 
\let\ftype@table\ftype@figure 
\makeatother

\DeclarePairedDelimiter\floor{\lfloor}{\rfloor}

\newcommand{\abs}[1]{\left| #1 \right|}

\renewcommand{\leq}{\leqslant}
\renewcommand{\geq}{\geqslant}
\newcommand{\E}{\mathbb{E}}
\newcommand{\K}{\mathcal{K}}

\begin{document}
\title{A conjectural asymptotic formula for multiplicative chaos in number theory}

\author{Daksh Aggarwal}
\address{
Department of Mathematics \\ 
Grinnell College\\ 
1115 8th Ave \# 3011\\
Grinnell, IA\\
USA \\
50112} 
\email{aggarwal2@grinnell.edu}

\author{Unique Subedi}
\address{
Department of Statistics \\ 
University of Michigan\\ 
1085 University Ave\\
323 West Hall\\
Ann Arbor, MI\\
USA \\
48109
} 
\email{subedi@umich.edu}

\author{William Verreault}
\address{
D\'{e}partement de Math\'{e}matiques et de Statistique\\ 
Universit\'{e} Laval\\ 
Qu\'ebec\\
QC\\
G1V 0A6 \\
Canada} 
\email{william.verreault.2@ulaval.ca}

\author{Asif Zaman}
\address{Department of Mathematics\\
University of Toronto \\
40 St. George Street, Room 6290 \\
Toronto, ON \\
Canada \\
M5S 2E4}
\email{zaman@math.toronto.edu}

\author{Chenghui Zheng}
\address{Department of Statistics\\
University of Toronto \\
100 St.George Street\\
Toronto, ON \\
Canada \\
M5S 3G3}
\email{chenghui.zheng@mail.utoronto.ca}

\maketitle

\begin{abstract} We investigate a special sequence of random variables $A(N)$ defined by an exponential power series with independent standard complex Gaussians $(X(k))_{k \geq 1}$. Introduced by Hughes, Keating, and O'Connell in the study of random matrix theory, this sequence relates to Gaussian multiplicative chaos (in particular ``holomorphic multiplicative chaos'' per Najnudel, Paquette, and Simm) and random multiplicative functions. Soundararajan and Zaman recently determined the order of $\E[|A(N)|]$.  By constructing an  algorithm to calculate $A(N)$ in $O(N^2 \log N)$ steps, we produce computational evidence that their result can likely be strengthened to an asymptotic result with a numerical estimate for the asymptotic constant.  We also obtain similar conclusions when $A(N)$ is defined using standard real Gaussians or uniform $\pm 1$ random variables. However, our evidence suggests that the asymptotic constants do not possess a natural product structure.
\end{abstract}

\section{Introduction}

Let $(X(k))_{k \geq 1}$ be a sequence of independent standard complex Gaussians; hence, the real and imaginary parts of $X(k)$ are independent real Gaussians with mean $0$ and variance $\frac{1}{2}$. Define the sequence of random variables $(A(N))_{N \geq 0}$ by the formal power series identity
\begin{equation}\label{power_series}
\exp\left(\sum_{k=1}^{\infty}\frac{X(k)}{\sqrt{k}}z^k\right) = \sum_{n=0}^{\infty} A(n) z^n.
\end{equation}
These random variables $A(N)$ naturally arise in several areas of probability and number theory.  As far as we are aware, they were first explicitly introduced by Hughes, Keating, and O'Connell \cite[(2.25)]{HughesKeatingOConnell-2000} in the context of random matrix theory.  Subsequent influential work of Fyodorov, Hiary, and Keating \cite{FyodorovHiaryKeating-2012,FyodorovKeating-2014} conjectured deep connections between random matrix theory, maxima of the Riemann zeta function on the critical line, and Gaussian multiplicative chaos. There is a vast literature  on each of these topics, so we shall refer the reader to some recent  surveys by Rhodes and Vargas \cite{RhodesVargas-2014}, Duplantier, Rhodes, Sheffield, and Vargas \cite{DuplantierRhodesSheffieldVargas-2017}, and Bailey and Keating \cite{BaileyKeating-2021}. On the probability side, Chhaibi and Najnudel \cite{ChhaibiNajnudel-2019} and Najnudel, Paquette, and Simm \cite{NajnudelPaquetteSimm-2020} studied the variables $A(N)$ (which they refer to as ``holomorphic multiplicative chaos'') to  establish direct links between random matrix theory and Gaussian multiplicative chaos. On the number theory side, Soundararajan and Zaman \cite{SoundararajanZaman-2021} studied $A(N)$ as a model problem for a  breakthrough of Harper \cite{harper_moments_2020} on the partial sums of random multiplicative functions.  

The existence of a limiting distribution for $A(N)$ is unknown, but it should presumably be closely related to the limiting distribution of the ``total mass of critical Gaussian multiplicative chaos'' studied by Duplantier et al. \cite{DuplantierRhodesSheffieldVargas2014u,DuplantierRhodesSheffieldVargas-2014} and Barral et al.  \cite{BarralKupiainenNikulaSaksmanWebb2015f} for example. Some progress towards the distribution of $A(N)$ has recently been made by estimating its moments. Notably, building on work of Diaconis and Gamburd \cite{DiaconisGamburd-2004} with magic squares (see also Gorodetsky \cite{gorodetskyMagicSquaresSymmetric2021}),  Najnudel, Paquette, and Simm \cite{NajnudelPaquetteSimm-2020} recently proved a beautiful formula: for positive integers $q \geq 1$,
 \[
 \mathbb{E}[ |A(N)|^{2q} ] = \#\{ q \times q \text{ squares with $\mathbb{Z}_{\geq 0}$ entries and all row and column sums equal to $N$} \}.
 \]
  For example, this elegant combinatorial identity implies that the $L^2$-moment satisfies
 \begin{equation}
  \mathbb{E}[ |A(N)|^{2} ] = 1.
  \label{eqn:L2} 	
 \end{equation}
 A proof of \eqref{eqn:L2} also appears in Soundararajan and Zaman \cite{SoundararajanZaman-2021} as a consequence of the cycle index formula for the symmetric group. Surprisingly, the order of the $L^1$-moment   is a bit smaller: there exists absolute positive constants $C_1$ and $C_2$ such that for $N \geq 2$,
\begin{equation}\label{eqn:L1}
\frac{C_1}{(\log N)^{1/4}} \leq \mathbb{E}[|A(N)|] \leq \frac{C_2}{(\log N)^{1/4}}.    
\end{equation}
This estimate (and a similar one for all the lower moments $\mathbb{E}[ |A(N)|^{2q} ]$ with $0 \leq q \leq 1$) was recently established by Soundararajan and Zaman \cite{SoundararajanZaman-2021} and the upper bound was proved independently by Najnudel, Paquette, and Simm \cite{NajnudelPaquetteSimm-2020}. In view of \eqref{eqn:L1}, it is natural to conjecture the following asymptotic formula for $\E[|A(N)|]$. 

\begin{conjecture} \label{conj:complex} Define $(A(N))_{N \geq 0}$ by \eqref{power_series} using a sequence $(X(k))_{k \geq 1}$ of independent standard complex Gaussians. 
	There exists a constant $C > 0$ such that 
	\[
     \mathbb{E}[|A(N)|] \sim \frac{C}{(\log{N})^{1/4}} \qquad  \text{as } N \to \infty.
	\]
\end{conjecture}
\noindent

Our interest in this conjecture primarily stems from its relationship to the theory of random multiplicative functions as outlined by Soundararajan and Zaman \cite{SoundararajanZaman-2021}. A random Steinhaus multiplicative function $f : \mathbb{N} \to \{|z|=1\}$ is obtained by picking an independent random variable $f(p)$ uniform on the unit circle $\{ |z|=1\}$ for each prime $p$ and extending it (completely) multiplicatively to all positive integers. Namely, if $n = p_1 \cdots p_{\ell}$ then $f(n) = f(p_1) \cdots f(p_{\ell})$. The (normalized) random partial sum $e^{-N/2} \sum_{n \leq e^N} f(n)$ parallels the random variable $A(N)$. Indeed, since we have that $\E[f(m)\overline{f(n)}] = \mathbf{1}_{m=n}$, it follows that the $L^2$-moment satisfies, for $N \geq 1$,
\[
\E\Big[ \Big| \sum_{n \leq e^N} f(n) \Big|^2 \Big] = \lfloor e^N \rfloor,
\]
which mirrors \eqref{eqn:L2}. As with $A(N)$, the limiting distribution of $e^{-N/2}\sum_{n \leq e^N} f(n)$ is not yet known to exist, but it should likely be related to the ``total mass of critical Gaussian multiplicative chaos''.  Harper \cite[Theorem 1 and Corollary 2]{harper_moments_2020} proved strong bounds for the tails of $\sum_{n \leq e^N} f(n)$ and amazingly showed that there exists absolute positive constants $C_1$ and $C_2$ such that for $N \geq 2$,
\[
\frac{C_1 e^{N/2}}{ (\log N)^{1/4} } \leq \E\Big[ \Big| \sum_{n \leq e^N} f(n) \Big| \Big] \leq \frac{C_2 e^{N/2}}{ (\log N)^{1/4} }.
\]
This inspired the proof of \eqref{eqn:L1} in \cite{SoundararajanZaman-2021} and also suggests a conjecture. Namely, there conjecturally exists an absolute constant $C > 0$ (not necessarily the same as in \cref{conj:complex}) such that 
\begin{equation}
\E\Big[ \Big| \sum_{n \leq e^N} f(n) \Big| \Big] \sim \frac{C e^{N/2}}{(\log N)^{1/4}} \qquad  \text{as } N \to \infty.
\label{eqn:asymptotic_integer}	
\end{equation}
In our view, evidence towards \cref{conj:complex} acts as indirect evidence for the above conjectural asymptotic formula. This view is additionally supported by the strong parallels between $A(N)$ and partial sums of random multiplicative functions over the polynomial ring $\mathbb{F}_q[t]$; see \cite{SoundararajanZaman-2021} for details. 

From a number theory perspective, it is reasonable to wonder whether the putative constant $C$  in  \eqref{eqn:asymptotic_integer} possesses an Euler product structure. That is, does there exist a sequence of complex numbers $(\beta_p)_{p}$ indexed by primes $p$  such that  $C$  in  \eqref{eqn:asymptotic_integer} satisfies
	\[
	C = \prod_{p} \beta_p \, ?
	\]
	If so, each local factor $\beta_p$ would presumably depend at most on the prime  $p$ and the distribution of $f(p)$. The parallels between \eqref{eqn:asymptotic_integer} and \cref{conj:complex} consequently prompt an informal question. 

\begin{question}\label{q:complex}
If \cref{conj:complex} holds, then does there exist a sequence of complex numbers $(\beta_k)_{k\geq 1}$ such that 
\begin{equation} \label{eqn:product}
    C = \prod_{k \geq 1} \beta_k,
\end{equation}
where each local factor $\beta_k$ depends at most on $k$ and the distribution of $X(k)$?
\end{question}

The purpose of this article is to computationally investigate the distribution of $A(N)$ and, in particular, test \cref{conj:complex} and \cref{q:complex} via  Monte Carlo simulations. To obtain estimates with reasonable precision, we must therefore efficiently calculate $A(N)$ for a large number of samples of the random sequence $(X(k))_{k \geq 1}$. This requirement is at the heart of our experimental pursuit and poses two key challenges. 

First, the sample size must be quite large to make conclusions with reasonable precision. By \eqref{power_series}, each $A(N)$ is defined in terms of the $N$ independent random variables $X(1), \dots, X(N)$ (see \cref{subsec:algorithm_description} for details), so we must compute $A(N)$ for a total number of independent samples that grows at least exponentially with $N$. This challenge is intrinsic to Monte Carlo simulations, so we do not attempt to address it. Second, a naive application of the power series identity \eqref{power_series} yields an expensive method for calculating $A(N)$ in terms of partitions of $N$ (see \cref{subsec:algorithm_description} for details). The number of partitions of $N$ is asymptotically $\frac{1}{4N\sqrt{3}} \exp( \pi \sqrt{2N/3})$. Even with the most efficient algorithms  to generate integer partitions, this brute force implementation appears to require a sub-exponential time and space complexity. This is prohibitive for our purposes. We instead devise an efficient algorithm to compute $A(N)$ for a single instance of $(X(k))_{k \geq 1}$.

\begin{theorem}\label{thm:algorithm}
Given any fixed sequence of complex numbers $(X(k))_{k \geq 1}$, define the sequence of complex numbers $(A(n))_{n \geq 0}$ by the formal power series identity \eqref{power_series}. For any $N \in \mathbb{N}$,  there exists an algorithm that computes the $N$ values $A(1),\dots,A(N)$ in  $O(N^2\log{N})$ time using $O(N)$ space.
\end{theorem}

Our algorithm does not generate partitions. Instead, it attains its efficiency by exploiting the recursive properties of integer partitions as well as the recursive structure of $A(N)$. See \cref{sec:Algorithm} for a description of our algorithm and the proof of \cref{thm:algorithm}. 

Equipped with \cref{thm:algorithm}, we can investigate variants of \cref{conj:complex} and \cref{q:complex}  for sequences $(X(k))_{k \geq 1}$ other than independent standard complex Gaussians, provided \eqref{eqn:L1} plausibly holds in those cases. The proof of \eqref{eqn:L1} in \cite{SoundararajanZaman-2021}  relies on the covariance structure of 
\[
\mathrm{Re} \sum_{e^n < k \leq e^{n+1}} \frac{X(k) e^{i k \theta}}{\sqrt{k}} 
\]
for $1 \leq n \leq \log N$ and $\theta \in [0,2\pi]$. For any fixed $n$ and $\theta$, note that these are independent real Gaussians with mean $0$ and variance close to $1/2$. These observations suggest \eqref{eqn:L1} is plausibly true for random variables $A(N)$ defined by \eqref{power_series} with sequences $(X(k))_{k \geq 1}$ that preserve this structure. Consequently, we expand our investigation to include all three of the following scenarios:
\begin{itemize}
    \item $(X(k))_{k  \geq 1}$ is a sequence of independent standard complex Gaussians (as before). 
    \item $(X(k))_{k  \geq 1}$ is a sequence of independent standard real Gaussians. 
    \item $(X(k))_{k  \geq 1}$ is a sequence of independent random variables  uniform on $\{\pm 1\}$.
\end{itemize}
For convenience, we refer to the latter as ``a sequence of $\pm1$ variables". These are mathematically and computationally simpler. This allows us to push our computations further, calculate conditional expectations more precisely, and actually exhaust the sample space of $A(N)$ for small values of $N$.   We therefore investigate variants of \cref{conj:complex} and \cref{q:complex}.
\begin{conjecture}\label{conj:other} \cref{conj:complex} also holds if $(X(k))_{k \geq 1}$ is a sequence of independent standard real Gaussians or a sequence of independent uniform $\{ \pm 1 \}$ random variables (each with a  possibly different constant $C$). 
\end{conjecture}

\begin{question}\label{q:other}
Do the constants $C$ from \cref{conj:other} have a product structure as in \cref{q:complex}?
\end{question}

We perform calculations for each type of random variable with $N = 2 \times 10^4$ and $5 \times 10^7$ samples. We report our conclusions in Sections \ref{sec:complex}, \ref{sec:real}, and \ref{sec:pm1}. Our data  supports \cref{conj:complex} and \ref{conj:other} in all cases. Table \ref{table:result_summary} lists the estimated values of the asymptotic constant for each of the three types of random variables. 
However, by considering conditional expectations  for $\pm 1$ variables, our computational evidence suggests the answer to Question \ref{q:other} (and hence \cref{q:complex}) is \textit{negative}.  The details of this investigation can be found in Section \ref{Euler_product_conjecture}.

\begin{table}
 \begin{tabular}{|c||c|} 
 \hline
Distribution of $X(k)$ & Estimated value of $C$  \\ [0.5ex] 
 \hline\hline
  Standard complex normal & 1.07 \\ 
 \hline
 Standard real normal & 0.957 \\ 
 \hline
Uniform on $\{\pm 1\}$ & 0.896 \\ 
 \hline
\end{tabular} ~\\[10pt]
\caption{Estimated values of the asymptotic constant $C$ in Conjectures \ref{conj:complex} and \ref{conj:other}.}
\label{table:result_summary}
\end{table}

 Finally, we could in principle carry out a similar computational study for a random Steinhaus multiplicative function $f$ over the integers to investigate the distribution of $\sum_{n \leq e^N} f(n)$ and \eqref{eqn:asymptotic_integer}.  Both $e^{-N/2} \sum_{n \leq e^N} f(n)$ and $A(N)$ have an $L^1$-moment that decays with rate $(\log N)^{-1/4}$. However, the number of primes $\leq e^N$ is asymptotically $e^N/N$, so the number of samples required for a Monte Carlo simulation of $e^{-N/2} \sum_{n \leq e^N} f(n)$ must grow at least exponentially with $e^N/N$. This requirement is substantially worse compared to $A(N)$ which only needs the number of samples to grow exponentially with $N$.  For instance, a calculation with a random multiplicative function $f$ analogous to the one we have performed for $A(N)$ would require us to calculate the partial sum $\sum_{n \leq e^N}f(n)$ with $e^N$ on the order of $10^{8685}$ and hence sample sequences with length $10^{8681}$. Moreover, we would need to efficiently compute $\sum_{n \leq e^N} f(n)$, but it is not clear to us whether this can be done in polynomial time with respect to $N$ as in \cref{thm:algorithm}.  Thus, given the computational resources required to carry out such a large-scale computation, it appears rather difficult to adequately investigate the corresponding conjecture and questions for random multiplicative functions. 

\medskip

\subsection*{Acknowledgements}
We are grateful to the Fields Institute for enabling this collaboration as part of their Undergraduate Summer Research Program and for providing financial support. This research was also supported by Compute Canada (\href{https://www.computecanada.ca}{\tt www.computecanada.ca}), where the majority of our computations were performed. We also thank Adam Harper for insightful comments and clarifications on the conjectural distribution of $A(N)$ and critical multiplicative chaos.

\section{Proof of Theorem \ref{thm:algorithm}}\label{sec:Algorithm}

\subsection{Description of the algorithm} \label{subsec:algorithm_description}

We follow \cite{SoundararajanZaman-2021} and introduce similar notation. A partition $\lambda$ is  a non-increasing sequence of non-negative integers $\lambda_1 \geq \lambda_2 \geq \cdots$ with $\lambda_n = 0$ from some point onwards. Let $|\lambda|$ be the sum of parts $\lambda_1 + \lambda_2 + \cdots$ and for an integer $k \geq 1$, let $m_k = m_k(\lambda)$ be the number of parts of $\lambda$ that are equal to $k$.  For a partition $\lambda$, define
\begin{equation}\label{ceq:3}
   a(\lambda) = a(\lambda; X) \coloneqq \prod_{k}\left(\frac{X(k)}{\sqrt{k}}\right)^{m_k}\frac{1}{m_k!} , 
\end{equation}
so that for $n \geq 1$,
\begin{equation} \label{eqn:partition_identity}
    A(n) = \sum_{|\lambda| = n} a(\lambda).
\end{equation}

For $k \geq 1$, let $A_k(n)$ be the contribution to $A(n)$ by partitions $\lambda$ such that $\lambda_1 = k$ and let $B_k(n)$ be the contribution to $A(n)$ by partitions $\lambda$ with $\lambda_1 \leq k$. In other words,
\begin{equation}\label{A_k_B_k}
   A_k(n) = \sum_{\substack{|\lambda| = n \\ \lambda_1 =k}} a(\lambda) \;\;\;
   \; \text{ and } \;\;\;\;\; B_k(n) =  \sum_{\substack{|\lambda|=n \\ \lambda_1\leq k}} a(\lambda).
\end{equation}

Now, consider a partition $|\lambda| = n$ such that $\lambda_1=k$. If $k$ has a multiplicity of $m$ in $\lambda$, then $\lambda$ is an extension of a unique partition $\lambda^{\prime}$ with $|\lambda'|= n-mk$ and $\lambda^{\prime}_1 \leq k-1$. In particular,  the partition $\lambda$ can be obtained by adding the part $k$ to the partition $\lambda^{\prime}$ exactly $m$ times. Since $\lambda_1 = k$, we note that $k$ can have a multiplicity of at most  $ \floor{\frac{n}{k}}$ in the partition $\lambda$. Thus, we deduce that
\[A_k(n) = \sum_{\substack{|\lambda| = n \\ \lambda_1=k}} a(\lambda) = \sum_{m=1}^{\floor{\frac{n}{k}}}\left(\frac{X(k)}{\sqrt{k}}\right)^m\frac{1}{m!}\;\sum_{\substack{|\lambda'|=  n-mk \\ \lambda_1' \leq k-1}} a(\lambda'),\]
which by \eqref{A_k_B_k} can be reformulated as
\begin{equation}\label{recursive_identity}
  A_k(n) =  \sum_{m=1}^{\floor{\frac{n}{k}}}\left(\frac{X(k)}{\sqrt{k}}\right)^m\frac{1}{m!}\;B_{k-1}(n-mk). 
\end{equation}
Since \eqref{A_k_B_k} is equivalent to
\[ B_{k-1}(n-mk) = \sum_{i=1}^{k-1} A_i(n-mk),\]
we can compute $A_k(n)$ for $1\leq n \leq N$ and $1\leq k \leq n$ recursively using \eqref{recursive_identity}. Finally, we compute the value of $A(n)$ using the identity
\[A(n) = \sum_{k=1}^{n}A_k(n).\]

\begin{remark} In \eqref{recursive_identity}, we note that the computation of $A_k (n)$ for a fixed pair $(k,n)$ only requires $B_{k-1}(n-mk)$ rather than individual $A_i(j)$ for all $0 \leq i \leq k-1$ and $1 \leq j \leq n-1$.  This observation allows us to improve the memory efficiency of our algorithm from $O(N^2)$ to $O(N)$ as we only have to store a vector of length $N$ instead of a matrix of size $N \times N$ while computing a sequence $(A(n))_{1 \leq n \leq N}$. 
\end{remark}

We now state the algorithm.
\begin{algorithm}[H]
\caption{Generate $(A(n))_{1 \leq n \leq N}$}
\label{alg:algorithm}
\begin{algorithmic}[1]
\REQUIRE $N \geq 1$, Sample sequence $(X(k))_{1 \leq k \leq N}$.
\ENSURE $(A(n))_{1 \leq n \leq N}$
\STATE Start at $k=1$. Compute $A_1(n)$ for all $1 \leq n \leq N$ and store it in an array $\mathcal{A}$. Use $A_0(0) = 1$.
\STATE  Using recursive identity  \eqref{recursive_identity}, compute $A_2(n)$ for all $2 \leq n \leq N$ and store it in an array $\mathcal{B}$. Update array $\mathcal{A}$ to the element-wise sum of arrays $\mathcal{A}$ and $\mathcal{B}$.

\STATE Repeat Step 2  for all $ k \leq N$. The last update of $\mathcal{A}$ ensures that $\mathcal{A}$ yields the required sequence $(A(n))_{1 \leq n \leq N}$ because $A_k(n)$ is summed over all partitions with the largest part $k$ for all $k \in \{1,2, \ldots, N\}$. 
\end{algorithmic}
\end{algorithm}

\noindent \textit{Remark}. In Step $2$,  notice that $n$ starts from $2$ because $1$ does not have a partition with largest part 2. Thus, for each $k$, we only have to compute $A_k(n)$ for $k \leq n \leq N$. Moreover, $B_{k-1}(n-mk)$ in \eqref{recursive_identity} is the $(n-mk)^{th}$ element of array $\mathcal{A}$.
\subsection{Computational complexity} 
Since we store the values of $A_{k-1}(n-mk)$, the number of operations required to compute a summand of identity \eqref{recursive_identity}
 is $O(1)$. For a fixed $k \leq n$, the number of steps required to compute $A_k(n)$ is 
 \[\sum_{m=1}^{\floor{\frac{n}{k}}}O(1) = O\Big(\frac{n}{k
 }\Big).\]
 Further, note that for a fixed $k$,  Steps $1$ and $2$ of Algorithm \ref{alg:algorithm}
compute $A_k(n)$ for all $k \leq n \leq N$. The number of steps required for such a computation is 
 \[\sum_{n=k}^{N}O\Big(\frac{n}{k}\Big) = O\Big(\frac{N^2}{k}\Big).\]
In Step 3, we repeat the same computation for all $1 \leq k \leq N$, and thus the total number of operations is
\[\sum_{k=1}^{N} O \Big(\frac{N^2}{k}\Big) = O(N^2\log{N}).\]
Thus, the algorithm runs in $O(N^2\log{N})$ time. The memory complexity of the algorithm is $O(N)$ because we only use two arrays of length $N$.

\subsection{Implementation}

This algorithm was implemented in the $\texttt{C++}$ programming language and the code can be accessed at \href{https://github.com/asif-z/multiplicative-chaos}{\tt https://github.com/asif-z/multiplicative-chaos}.  Our computations have been carried out on the clusters of Compute Canada. To achieve large-scale Monte Carlo simulations, we utilized multi-threading via \texttt{OpenMp} and \texttt{MPI}. 

We computed the averages of the values of $\abs{A(n)}$ using $S$ independent samples of the sequence $(X(k))_{1 \leq k \leq N}$ for each $1 \leq n \leq N$. To be clear, once a sample of the sequence $(X(k))_{1 \leq k \leq N}$ was drawn, it was used to compute the entire sequence $(A(n))_{1 \leq n \leq N}$. We present the results of this computation with $N= 2\times 10^4$ and $S = 5\times 10^7$ for all three types of random variables. To convey an idea of the computational resources utilized, our simulation for real standard normal variables with $N = 2\times 10^4$ consumed about 2 core-years using an average processor clock speed of 2.3 GHz.

\section{Results with Standard Complex Gaussians}  \label{sec:complex}
 Let $(X(k))_{k \geq 1}$ be  a sequence of independent standard complex Gaussians, and recall that $A(N)$ is defined by the power series identity \eqref{power_series}. For $N = 2 \times 10^4$, the distribution of the normalized variable $A(N)\cdot (\log(N))^{1/4}$ is displayed in Figure \ref{fig:distribution_clx} and the distribution of $|A(N)|\cdot (\log(N))^{1/4}$  is displayed in Figure \ref{fig:distribution_norm_clx}. Each figure was generated by Monte Carlo simulation with $4\times 10^6$ standard complex Gaussian sample sequences. While these do not address our central questions, we have included these visuals to provide an informal sense for the scale of our computation. 

 \begin{figure}
\centering

 \includegraphics[scale =0.75]{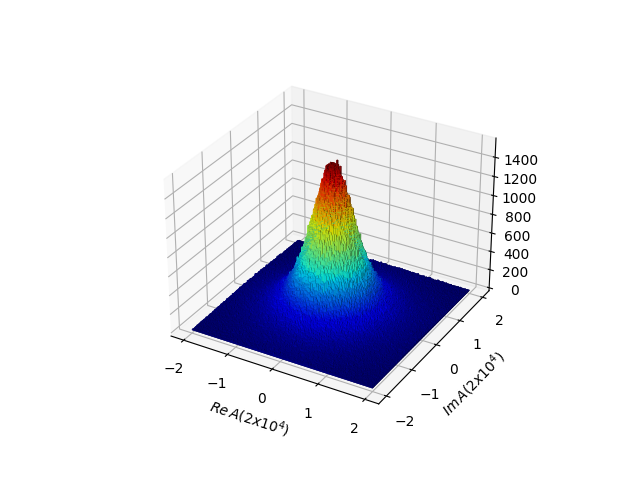}
    \caption{Distribution of $A(2\times 10^4)\cdot (\log(2\times 10^4))^{1/4}$ using $4\times 10^6$ complex Gaussian samples.}
    \label{fig:distribution_clx}

\end{figure}

 \begin{figure}
\centering

 \includegraphics[scale =0.6]{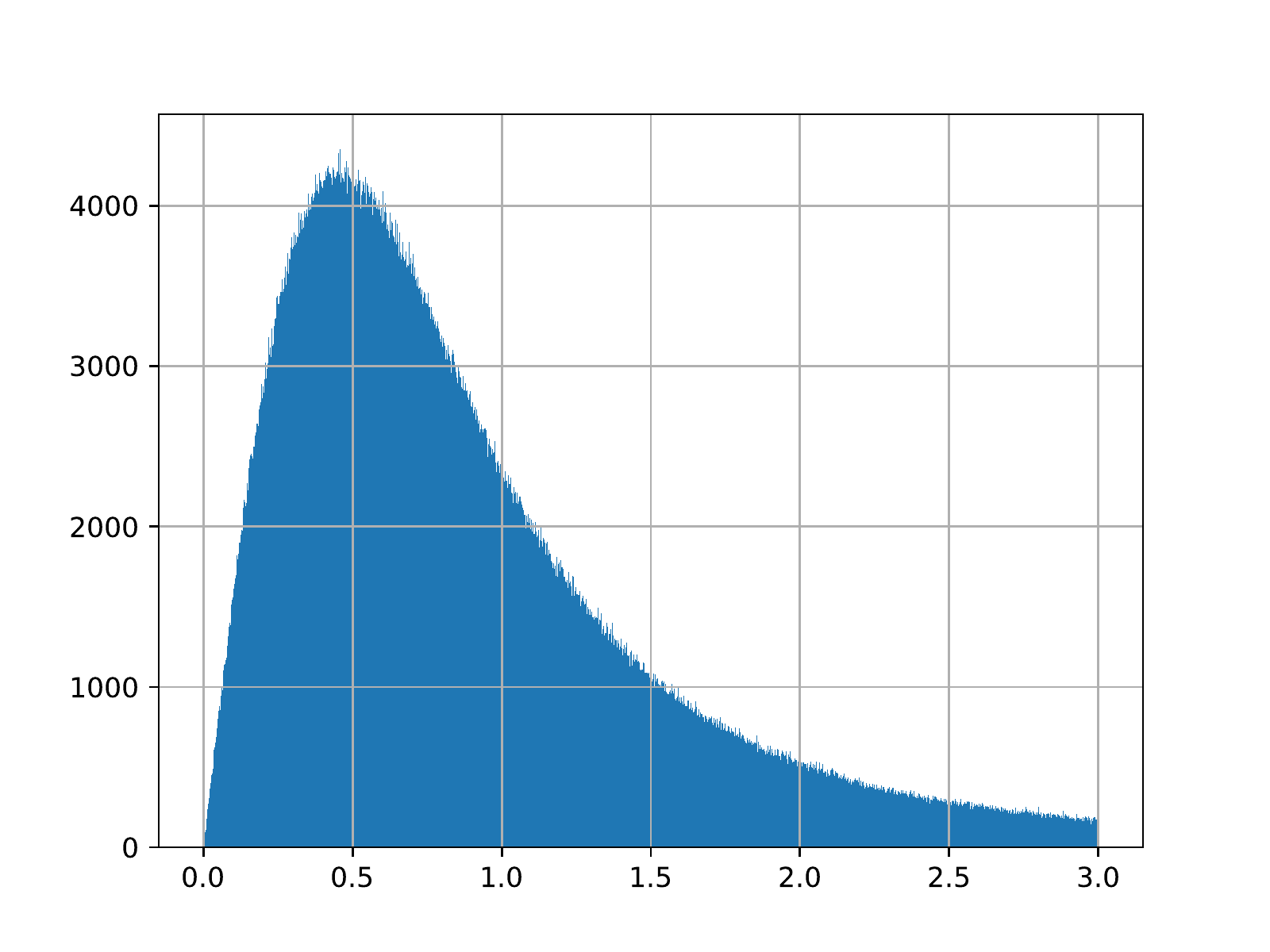}
    \caption{Distribution of $|A(2\times 10^4)|\cdot (\log(2\times 10^4))^{1/4} $ using $4\times 10^6$ complex Gaussian samples.}
    \label{fig:distribution_norm_clx}

\end{figure}
 
 For any $N \geq 1$, let $[|A(N)|]_{S}$ denote the sample mean of $|A(N)|$ generated with $S$ samples. In other words, if we let $\mathcal{S}$ denote a set of $S = |\mathcal{S}|$ sample sequences, then
 \[[|A(N)|]_{S} = \frac{1}{S}\sum_{X\in \mathcal{S}} |A(N;X)|.\]

\subsection{Analysis of \cref{conj:complex} for standard complex Gaussians}\label{subsec:analysis_complex}

To numerically verify Conjecture \ref{conj:complex}, we want to see that the product  $[|A(N)|]_S \cdot (\log N)^{1/4}$ approaches a constant for a sufficiently large $N$. For a sample mean generated with $ 5 \times 10^7$ samples, Figure \ref{fig:constant_complex_Gaussians} and Table \ref{table:constant_complex_Gaussians} show how $[|A(N)|]_{5 \times 10^7} \cdot (\log{N})^{1/4}$ changes for $10^3 \leq N \leq 20\times 10^3$. They both support the notion that $[|A(N)|]_{5 \times 10^7}  \cdot (\log{N})^{1/4}$ (and hence $\E[|A(N)|]$) is approaching a constant, thus providing evidence in favor of \cref{conj:complex}. From Table \ref{table:constant_complex_Gaussians}, the sample mean  $[|A(N)|]_{5 \times 10^7}  \cdot (\log{N})^{1/4}$  is fairly stable for $10^4 \leq N \leq 2 \times 10^4$ in the second decimal place with flucations in the third decimal place. Thus, $C \approx 1.07$ is the best estimate of the asymptotic constant in \cref{conj:complex} that our data provides. 

To obtain a more precise estimate of $C$, we would need to compute $[|A(N)|]_{S} $ for a much larger $N$. Unfortunately, generating data for larger values of $N$ requires computational resources beyond our availability, especially because a larger $N$ would need a larger sample size. 

\begin{figure}
\centering

 \includegraphics[scale =0.6]{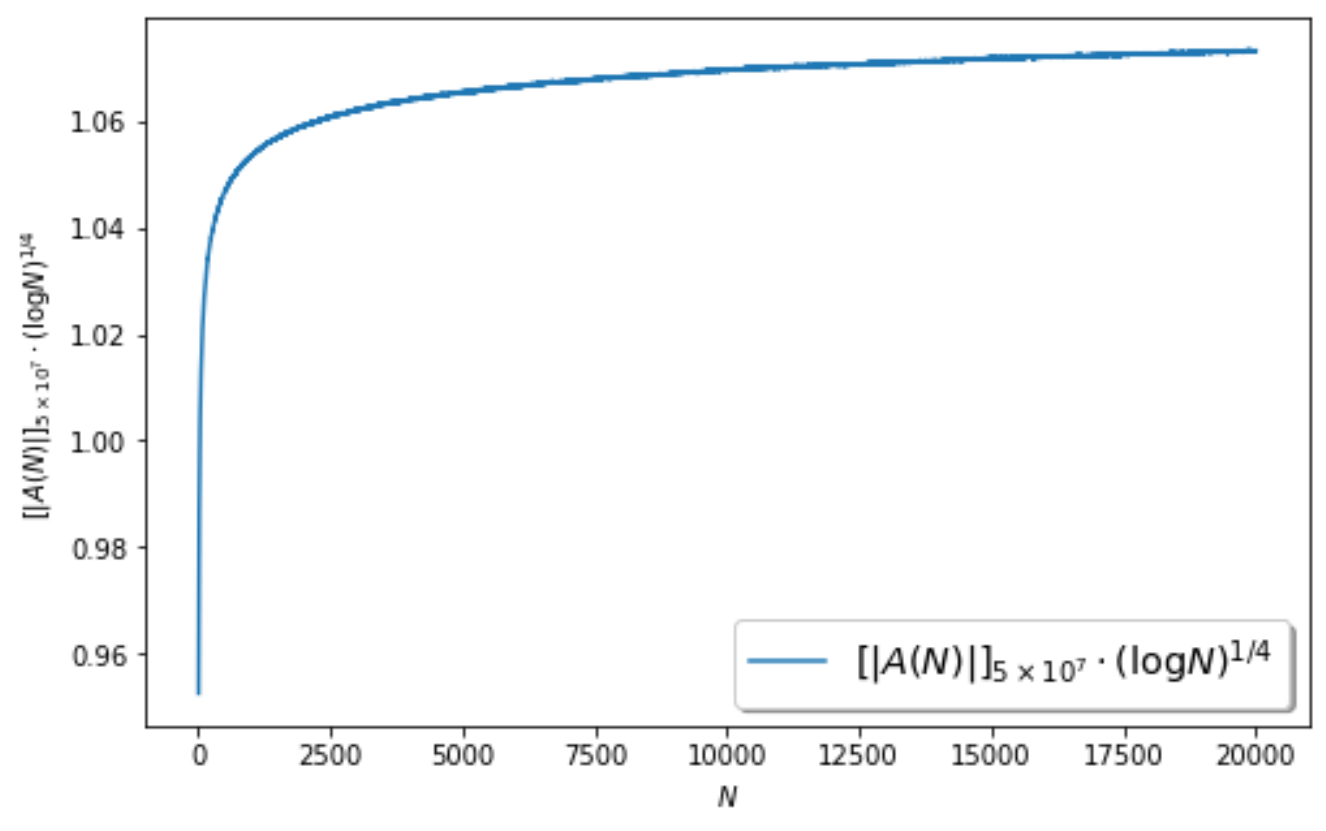}
    \caption{$[|A(N)|]_{5 \times 10^7} \cdot (\log{N})^{1/4}$ for $10 \leq N \leq 2\times10^4$.} 
    \label{fig:constant_complex_Gaussians}

\end{figure}

\begin{table}
\centering
 \begin{tabular}{| c | c || c | c |} 
 \hline 
 \\[-1em]
 $N $ & $[|A(N)|]_{5 \times 10^7}  \cdot (\log{N})^{1/4}$  & $N$ & $[|A(N)|]_{5 \times 10^7}  \cdot (\log{N})^{1/4}$ \\
 \hline \hline
 $1 \times 10^3$ & 1.0533 & $11 \times 10^3$ & 1.0700 \\
 \hline
 $2 \times 10^3$ & 1.0591 & $12 \times 10^3$ & 1.0706\\
 \hline
 $3 \times 10^3$ & 1.0620 & $13 \times 10^3$ & 1.0709 \\
  \hline
 $4 \times 10^3$ & 1.0640 & $14 \times 10^3$ & 1.0713 \\
 \hline
 $5 \times 10^3$ & 1.0653 & $15 \times 10^3$ & 1.0717 \\
 \hline 
  $6 \times 10^3$ & 1.0665 & $16 \times 10^3$ & 1.0721 \\
 \hline
 $7 \times 10^3$ & 1.0674 & $17 \times 10^3$ & 1.0723 \\
 \hline
 $8 \times 10^3$ & 1.0683 & $18 \times 10^3$ & 1.0726\\
  \hline
 $9 \times 10^3$ & 1.0690 & $19 \times 10^3$ & 1.0728 \\
 \hline
 $10 \times 10^3$ & 1.0695 & $20 \times 10^3$ & 1.0732 \\
 \hline
\end{tabular}
\caption{Values of $[|A(N)|]_{5 \times 10^7}  \cdot (\log{N})^{1/4}$ for $N \in \{i \times 10^3: 1 \leq i \leq 20\}$. }
\label{table:constant_complex_Gaussians}
\end{table}

\subsection{Robustness and sample size} \label{subsec:robust_complex} We share two results that suggest $[|A(N)|]_{5 \times 10^7} $ estimates $\mathbb{E}[\abs{A(N)}]$ accurately up to 3 decimal places with a high degree of certainty for all $1 \leq N \leq 2 \times 10^4$.  

First, we fix $N$ and study $[|A(N)|]_{S}$ as $S$ increases. Figure \ref{fig:smp_clx} shows the deviation of $[|A(2 \times 10^4)|]_{S}$ from the final sample mean for the last $10^6$ samples in a total sample size of $ 4\times 10^6$.  We see that for all $ 3\times 10^6 < S \leq 4\times 10^6$,
\begin{equation}
\label{eq:bounded_deviation}
     \abs{[|A(2\times 10^4)|]_{S}-[|A(2\times 10^4)|]_{4\times 10^6}} < 10^{-3}.
\end{equation}
 Even with a smaller sample size of $4\times 10^6$, the sample mean $[|A(2\times 10^4)|]_{S}$ appears to have stabilized in the first 3 decimal places, and indeed we find that \[\abs{[|A(2\times 10^4)|]_{5\times 10^7}- [|A(2\times 10^4)|]_{4\times 10^6}} < 10^{-3}.\]
This suggests that the data presented in Table \ref{table:constant_complex_Gaussians} is statistically significant up to 3 decimal places. 

\begin{figure}
\centering
\includegraphics[scale =0.6]{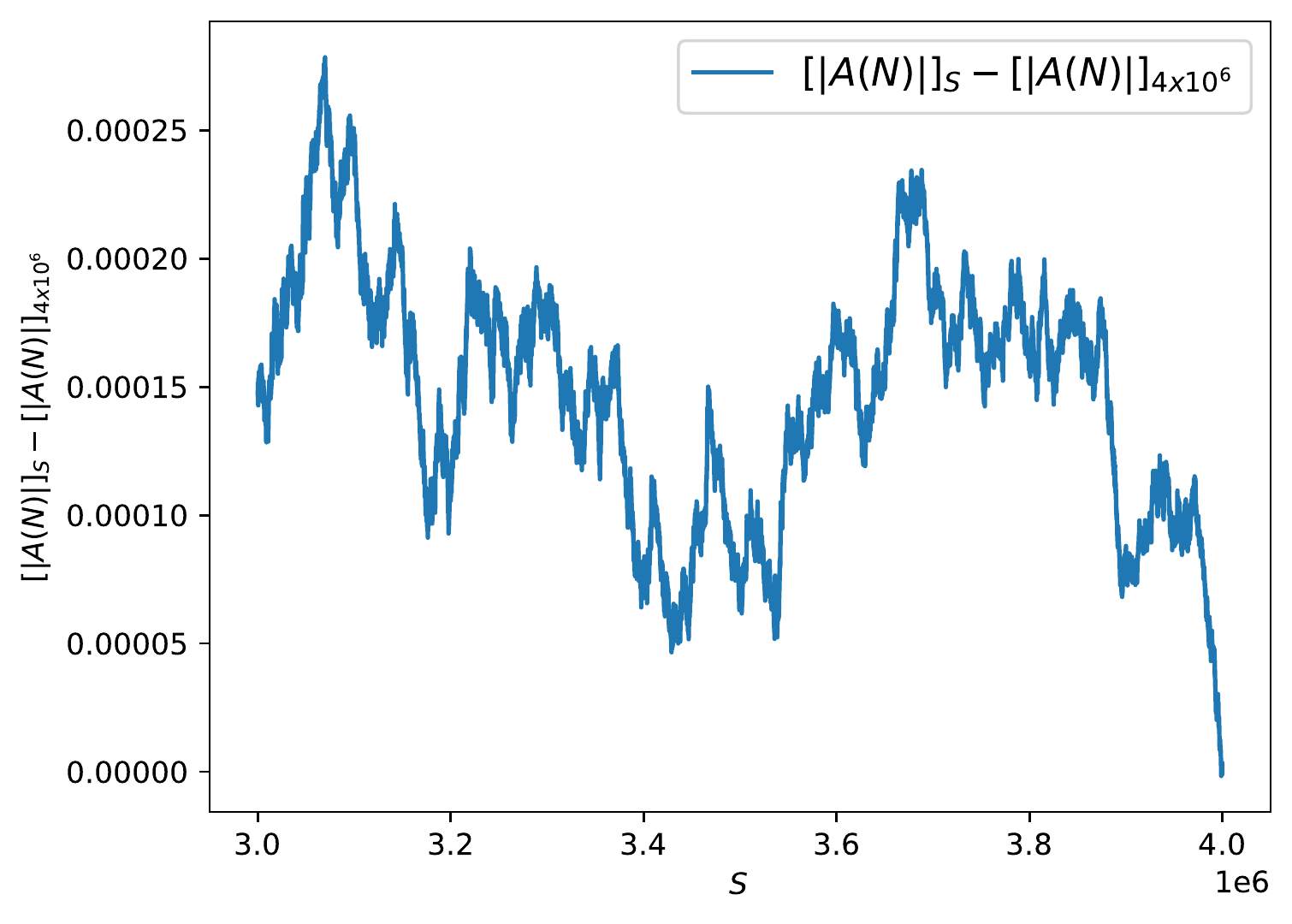}
\caption{ $[|A(N)|]_{S}-[|A(N)|]_{4 \times 10^6}$ for $N = 2\times 10^4$ and $3\times 10^6 <  S\leq 4 \times 10^6$ standard complex Gaussian samples.}
\label{fig:smp_clx}
\end{figure}

Next, for a fixed $S$, we study the deviation among sample means $[|A(N)|]_S$,  generated with 6 independent sets of samples sequences of the same size $S$. That is, we produced six different data sets with $S= 10^7$ samples and calculated the corresponding sample mean $[|A(N)|]_{S}$ for $1 \leq N \leq 2 \times 10^4$. For each $N$, let $[|A(N)|]_{10^7}^{(i)}$ denote the sample mean of $|A(N)|$ from the $i^{th}$ data set where $1 \leq i \leq 6$. We study the absolute deviation of $[|A(N)|]_{10^7}^{(i)}$ from $[|A(N)|]_{5 \times 10^7}$, which we treat as our benchmark. Denote this absolute deviation by $\delta_i$. That is, for each $i$ and $N$, define
\begin{equation}\label{absolute_deviation}
   \delta_i(N) = \Big|[|A(N)|]_{5 \times 10^7} - [|A(N)|]_{10^7}^{(i)}\Big|. 
\end{equation}
Figure \ref{fig:deviation_complex_Gaussians} shows the values of $\delta_i(N)$ for $ 1 \leq i \leq 6$ and $1 \leq N \leq 2 \times 10^6$ and Table \ref{table:deviation_complex_Gaussians} provides more detailed statistics for $\delta_i$. We observe that, except for a few points in the third data set $(i=3)$,  $\delta_i(N) < 10^{-3}$ for all $1 \leq i \leq 6$ and $1 \leq N \leq  2 \times 10^{4}$.
So, broadly speaking, we can infer that all six data sets are essentially identical up to three decimal digits. Since each of these data sets are generated with a sample size of $10^7$,  this strongly suggests that, for all $1 \leq N \leq 2 \times 10^4$, the sample mean $[|A(N)|]_{10^7}$ estimates  $\mathbb{E}[|A(N)|]$ accurately up to three decimal digits  with high degree of certainty. Therefore, we can effectively conclude that $[|A(N)|]_{5 \times 10^7} $, used in our analysis in the preceding subsection, is significant up to three decimal digits.

\begin{figure}[h]
\centering
\includegraphics[scale =0.6]{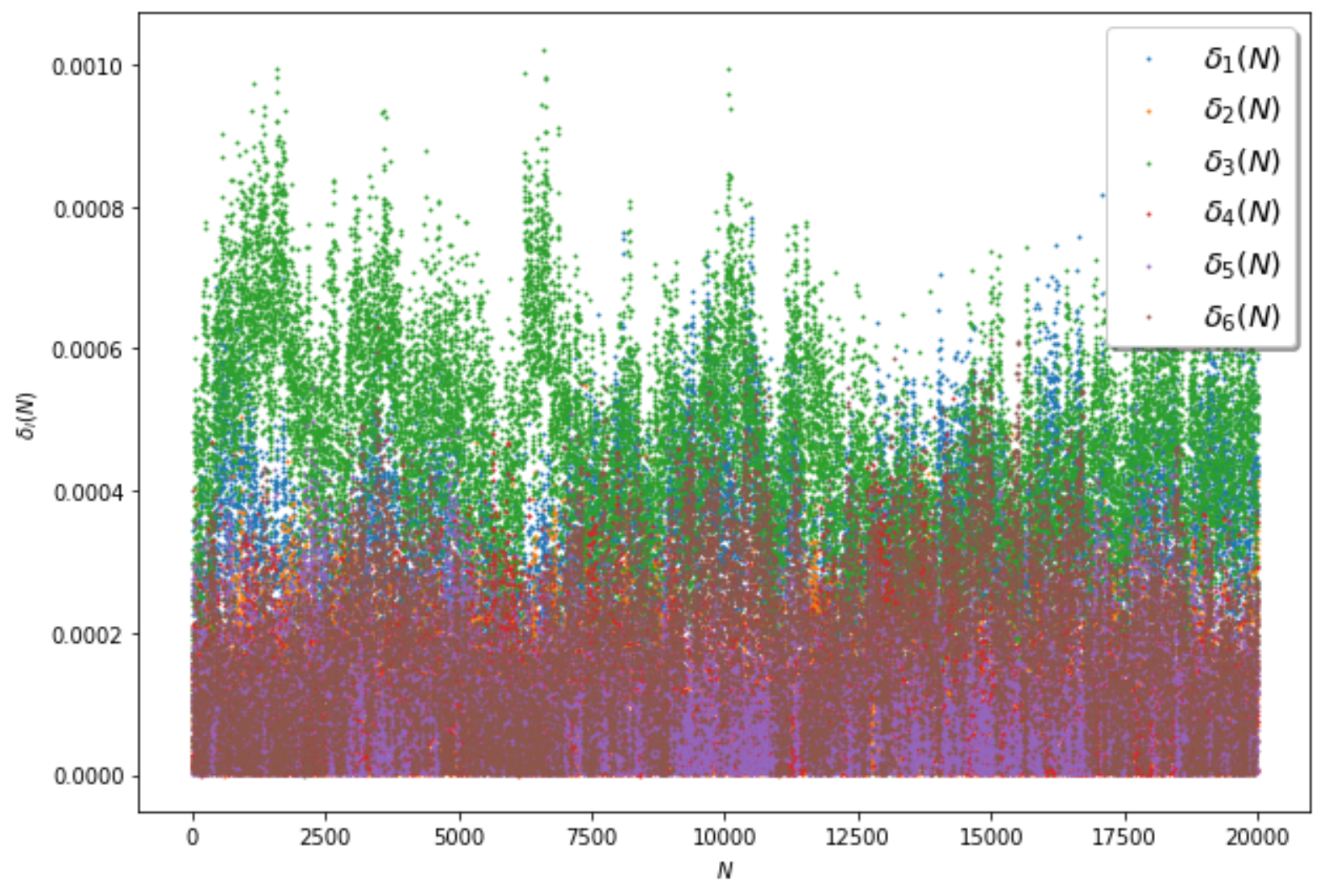}
    \caption{$\delta_i(N)$ for $1 \leq i \leq 6$ and $1 \leq N \leq 2 \times 10^4$.}
    \label{fig:deviation_complex_Gaussians}

\end{figure}

 \begin{table}[h]
 \begin{tabular}{|c | c | c | c |} 
\hline 
\\[-1em]

$i$ & $\text{Average of}\,\{\delta_i(N)\}_{N=1}^{2 \times 10^4}$& $\text{Maximum of}\,\{\delta_i(N)\}_{N=1}^{2 \times 10^4}$ & $\%$ of $N'$s with  $\delta_i(N) \geq 10^{-3}$ \\ 
 \hline \hline

1 & 0.0003 &  0.0008 & 0.0 \\
\hline
2 & 0.0001 &  0.0005 & 0.0 \\
\hline
3 & 0.0005 &  0.0010 & 0.005 \\
\hline
4 & 0.0001 &  0.0005 & 0.0 \\
\hline
5 & 0.0001 &  0.0005 & 0.0 \\
\hline
6 & 0.0002 &  0.0006 & 0.0 \\
\hline

\end{tabular}
\caption{Descriptive statistics for $\delta_i(N)$ for each $1 \leq i \leq 6$.}
\label{table:deviation_complex_Gaussians}
\end{table}

\section{Results with Standard Real Gaussians} \label{sec:real}

Let $(X(k))_{ k \geq 1}$ be a sequence of independent standard real Gaussians. Define the sequence of random variables 
$(A(N))_{N \geq 0}$ by the power series identity \eqref{power_series}, and let $[|A(N)|]_S $ denote the sample mean of $|A(N)|$ generated with $S$ samples. Figure \ref{fig:distribution_real} shows the distribution of $|A(2\times 10^4)|\cdot (\log(2\times 10^4))^{1/4}$ for a Monte Carlo simulation using $4\times 10^6$ real Gaussian sample sequences.

 \begin{figure}[h]
\centering
 \includegraphics[scale =0.6]{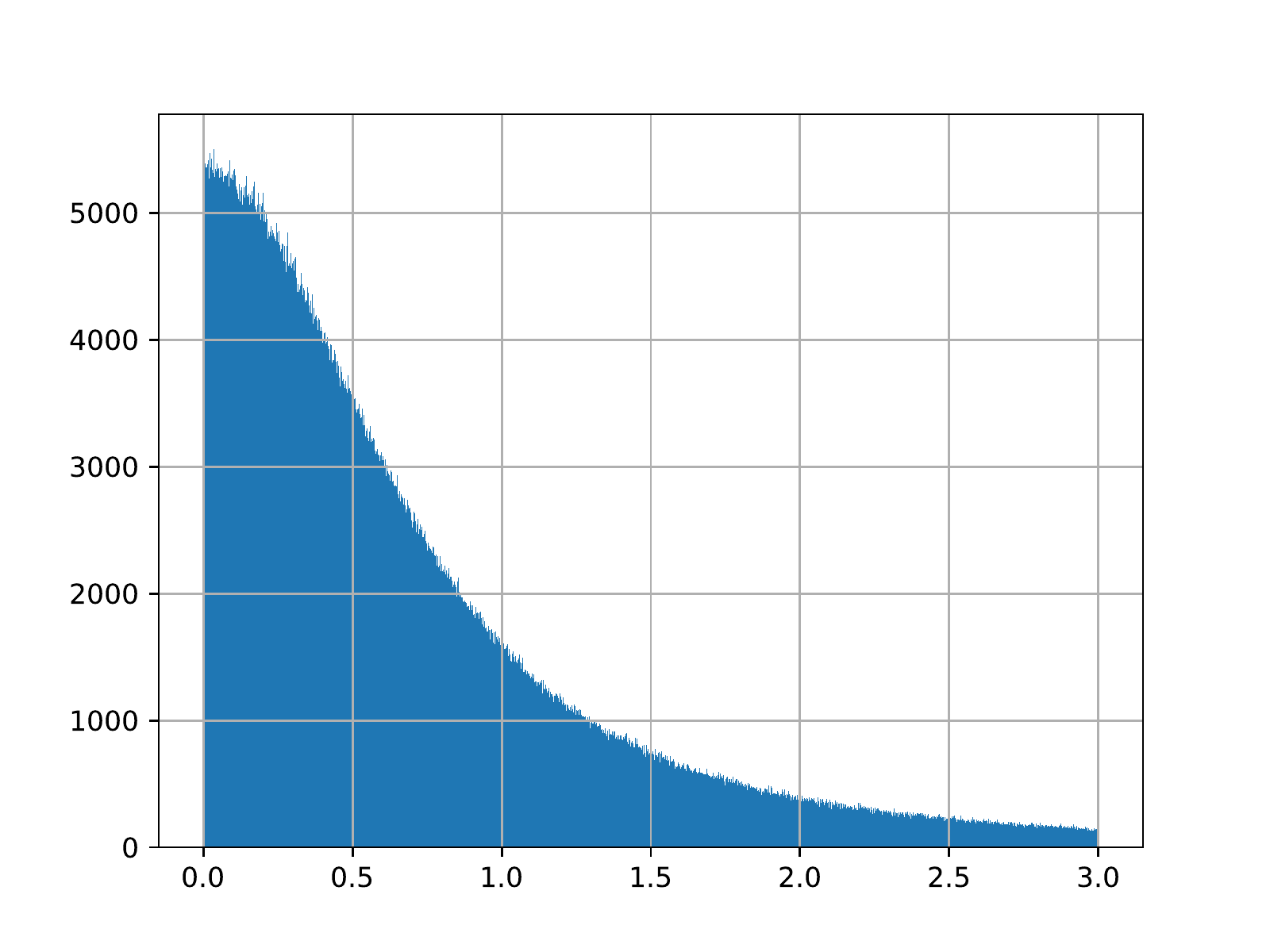}
    \caption{Distribution of $|A(2\times 10^4)|\cdot(\log(2\times 10^4))^{1/4}$ using $4\times 10^6$ real Gaussian samples.}
    \label{fig:distribution_real}
\end{figure}

\subsection{Analysis of \cref{conj:other} for standard real Gaussians} \label{subsec:analysis_real}

Like in \cref{sec:complex}, we approximate $\E[|A(N)|]$ with the sample mean $ [|A(N)|]_{5 \times 10^7}$ for $10 \leq N \leq 2 \times 10^4$. From Figure \ref{fig:constant_real_Gaussians} and Table \ref{table:constant_real_Gaussians}, we see that $[|A(N)|]_{5 \times 10^7}  \cdot (\log{N})^{1/4}$  is approaching a constant. Thus, our data  suggests that \cref{conj:other} holds for standard real Gaussians and that $C \approx 0.957$ based on Table \ref{table:constant_real_Gaussians}.

\begin{figure}[h]
\centering

 \includegraphics[scale =0.6]{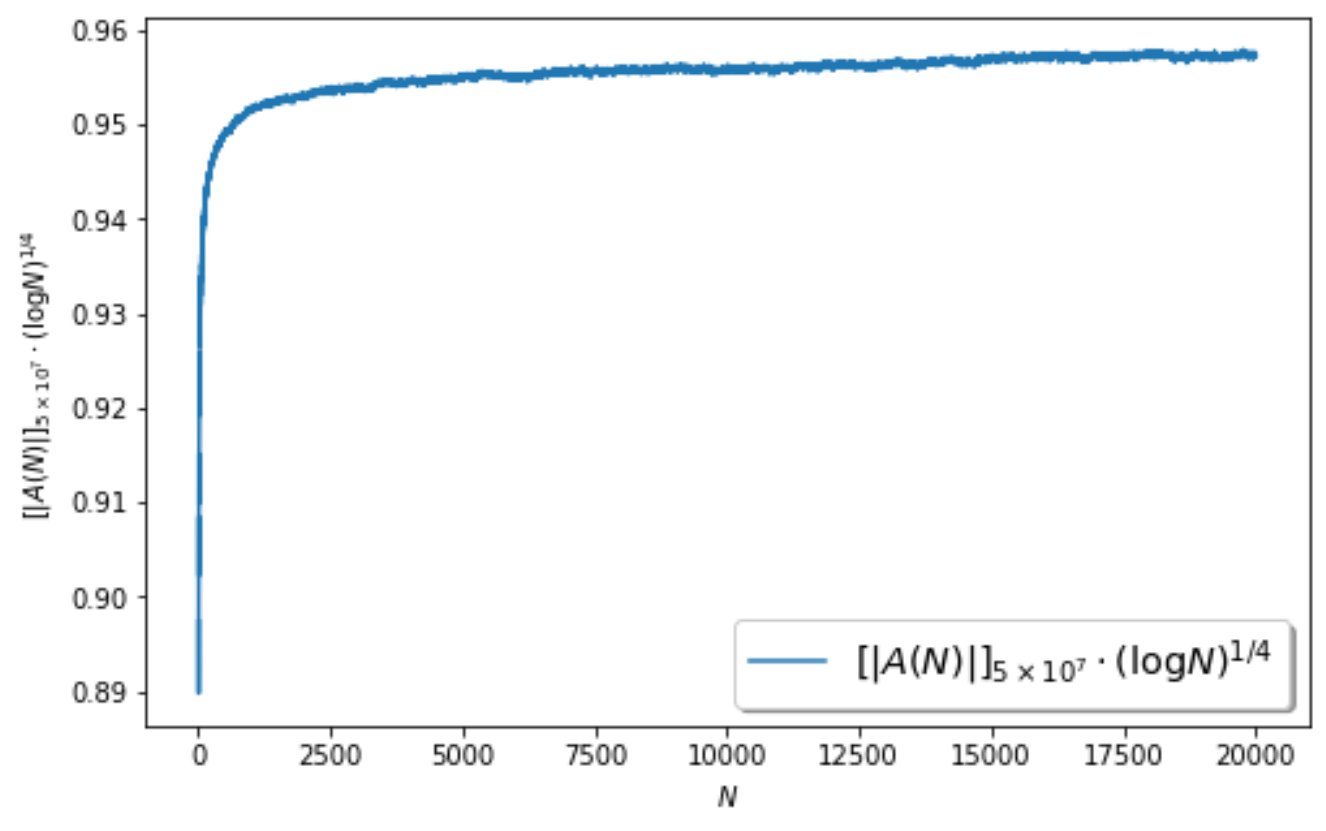}
    \caption{$[|A(N)|]_{5 \times 10^7} \cdot (\log{N})^{1/4}$ for $10 \leq N \leq 2\times10^4$.}
    \label{fig:constant_real_Gaussians}

\end{figure}

\begin{table}[h]
\centering
 \begin{tabular}{| c | c || c | c |} 
 \hline 
 \\[-1em]
 $N$ & $[|A(N)|]_{5 \times 10^7}  \cdot (\log{N})^{1/4}$  & $N$ & $[|A(N)|]_{5 \times 10^7}  \cdot (\log{N})^{1/4}$ \\
 \hline \hline
 $1 \times 10^3$ & 0.9517 & $11 \times 10^3$ & 0.9562 \\
 \hline
 $2 \times 10^3$ & 0.9528 & $12 \times 10^3$ & 0.9558 \\
 \hline
 $3 \times 10^3$ & 0.9543 & $13 \times 10^3$ & 0.9564\\
  \hline
 $4 \times 10^3$ & 0.9546 & $14 \times 10^3$ & 0.9566 \\
 \hline
 $5 \times 10^3$ & 0.9547 & $15 \times 10^3$ & 0.9570 \\
 \hline 
 $6 \times 10^3$ & 0.9552 & $16 \times 10^3$ & 0.9571  \\
 \hline
 $7 \times 10^3$ & 0.9555 & $17 \times 10^3$ & 0.9572 \\
 \hline
 $8 \times 10^3$ & 0.9559 & $18 \times 10^3$ & 0.9574 \\
  \hline
 $9 \times 10^3$ & 0.9561 & $19 \times 10^3$ & 0.9572 \\
 \hline
 $10 \times 10^3$ & 0.9559 & $20 \times 10^3$ & 0.9572 \\
 \hline
\end{tabular}
~\\
\caption{Values of $[|A(N)|]_{5 \times 10^7}  \cdot (\log{N})^{1/4}$ for $N \in \{i \times 10^3: 1 \leq i \leq 20\}.$ }
\label{table:constant_real_Gaussians}
\end{table}

 Notice that this is more precise than our estimation for the asymptotic constant in case of complex Gaussians. Recall that our data for complex Gaussians in \cref{subsec:analysis_complex} suggested that the values $[|A(N)|]_{5 \times 10^7} \cdot (\log N)^{1/4}$ in Table \ref{table:constant_complex_Gaussians} (and hence $\mathbb{E}[|A(N)|]$ itself)  had only converged to within 2 decimal places for $10^4 \leq N \leq 2 \times 10^4$. In this case, however,   the values $[|A(N)| ]_{5 \times 10^7}\cdot (\log N)^{1/4}$ in Table \ref{table:constant_real_Gaussians} are fairly stable for $10^4 \leq N \leq 2 \times 10^4$ in the third decimal place with fluctuations mostly in the fourth decimal place.

\subsection{Robustness and sample size} \label{subsec:robust_real} As before, we share two tests that suggests the sample mean $[|A(N)|]_{5 \times 10^7} $ estimates $\mathbb{E}[\abs{A(N)}]$ accurately up to 3 decimal places with a high degree of certainty for all $1 \leq N \leq 2 \times 10^4$. 

First, we fix $N = 2 \times 10^4$ and vary $S$ to study the deviation of $[|A(N)|]_{S}$ from the final sample mean  $[|A(N)|]_{4 \times 10^6}$ in Figure \ref{fig:smp_real}.  We observe that \eqref{eq:bounded_deviation} holds here too, so our data for the sample means $[|A(N)|]_{5\times 10^7}$, and thus Table \ref{table:constant_real_Gaussians}, is significant up to $3$ decimal places with high certainty.

\begin{figure}[h]
\centering
\includegraphics[scale =0.6]{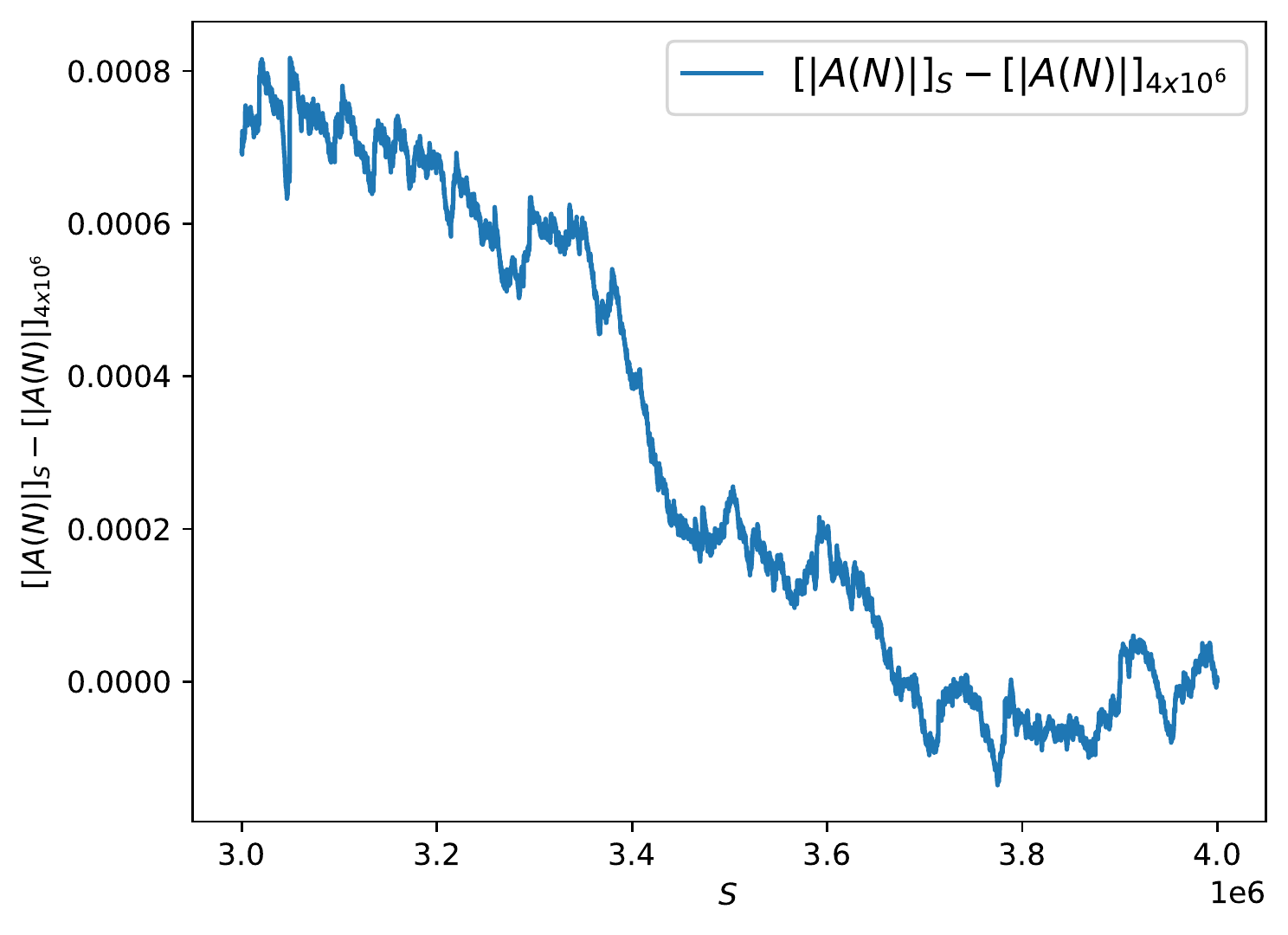}
\caption{ $[|A(N)|]_{S}-[|A(N)|]_{4 \times 10^6}$ for $N = 2\times 10^4$ and $3\times 10^6 <  S\leq 4 \times 10^6$.}
\label{fig:smp_real}
\end{figure}

Second, we study the deviation among 6 different sets of independent samples. Each data set calculates the sample mean $[|A(N)|]_{10^7}$ for $1 \leq N \leq 2 \times 10^4$. Adopting the notation from \cref{subsec:robust_complex}, let $[|A(N)|]_{10^7}^{(i)}$ denote the sample mean of $|A(N)|$ from the $i^{\text{th}}$ data set where $1 \leq i \leq 6$. Using $[|A(N)|]_{5 \times 10^7}$ as our benchmark, define $\delta_i(N)$ by \eqref{absolute_deviation} as the absolute value of the deviation. Figure \ref{fig:deviation_real_Gaussians} shows the values of $\delta_i(N)$ for $ 1 \leq i \leq 6$ and $1 \leq N \leq 2 \times 10^6$, and Table \ref{table:deviation_real_Gaussians} presents associated statistics. They show that $\delta_i(N) < 10^{-3}$ for almost all $N$ for $i \in \{1,2,4\}$.  However,  $\delta_i(N) \geq 10^{-3}$ for a considerable proportion of $N'$s for $i \in \{3, 5, 6\}$.  This suggests that we need a sample size larger than $10^{7}$ to recover three decimal points of $\mathbb{E}[|A(N)|]$ with high degree of certainty. Nevertheless, the average deviation is less than $10^{-3}$ and the maximum is less than $3 \times 10^{-3}$ in all six data sets. Thus, when the sample size is substantially increased from $10^7$ to $5 \times 10^{7}$, it is reasonable to conclude that the sample mean $[|A(N)|]_{5 \times 10^7}$ estimates $\mathbb{E}[|A(N)|]$ up to three decimal places with a much higher degree of certainty for all $1 \leq N \leq 2\times 10^4$.

\begin{figure}[h]
\centering
\includegraphics[scale =0.6]{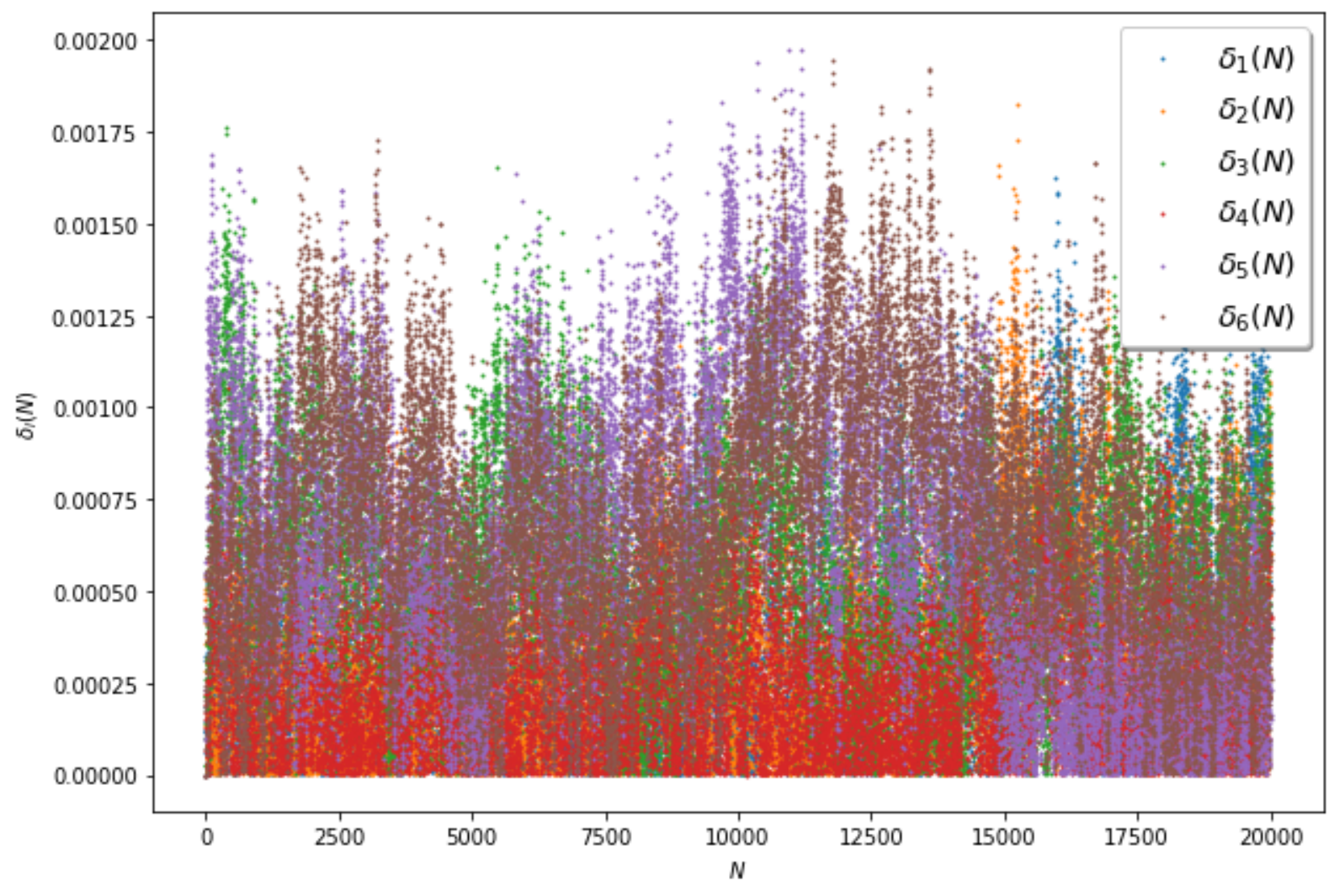}
    \caption{$\delta_i(N)$ for $1 \leq i \leq 6$ and $1 \leq N \leq 2 \times 10^4$.}
    \label{fig:deviation_real_Gaussians}

\end{figure}

\begin{table}[h]
\begin{tabular}{|c | c | c | c |} 
\hline 
\\[-1em]

$i$ & $\text{Average of}\,\{\delta_i(N)\}_{N=1}^{2 \times 10^4}$& $\text{Maximum of}\,\{\delta_i(N)\}_{N=1}^{2 \times 10^4}$ & $\%$ of $N'$s with  $\delta_i(N) \geq 10^{-3}$ \\ 
 \hline \hline

1 & 0.0003& 0.0016 & 1.780 \\
\hline
2 & 0.0003 & 0.0018 & 1.385 \\
\hline
3 & 0.0005 & 0.0018 & 6.375 \\
\hline
4 & 0.0003 & 0.0011 & 0.055 \\
\hline
5 & 0.0006 & 0.0020 & 18.74\\
\hline
6 & 0.0007 & 0.0019 & 19.69 \\
\hline

 \hline

\end{tabular} 
\medskip
\caption{Descriptive statistics for $\delta_i(N)$ for each $1 \leq i \leq 6$.}
\label{table:deviation_real_Gaussians}
\end{table}

\section{Results with \texorpdfstring{$\pm 1$}{} Random Variables}\label{sec:pm1}

Let $(X(k))_{k  \geq 1}$ be a sequence of independent   random variables uniform on $\{ \pm 1\}$ and, as usual, define $(A(N))_{N \geq 0}$  by the power series identity \eqref{power_series}. Again, $[|A(N)|]_S $ denotes the sample mean of $|A(N)|$ generated with $S$ samples. Figure \ref{fig:distribution_pm1} shows the distribution of $A(2\times 10^4)\cdot (\log(2\times 10^4))^{1/4}$ for a Monte Carlo simulation using $4\times 10^6$ real Gaussian sample sequences.

 \begin{figure}
\centering
 \includegraphics[scale =0.6]{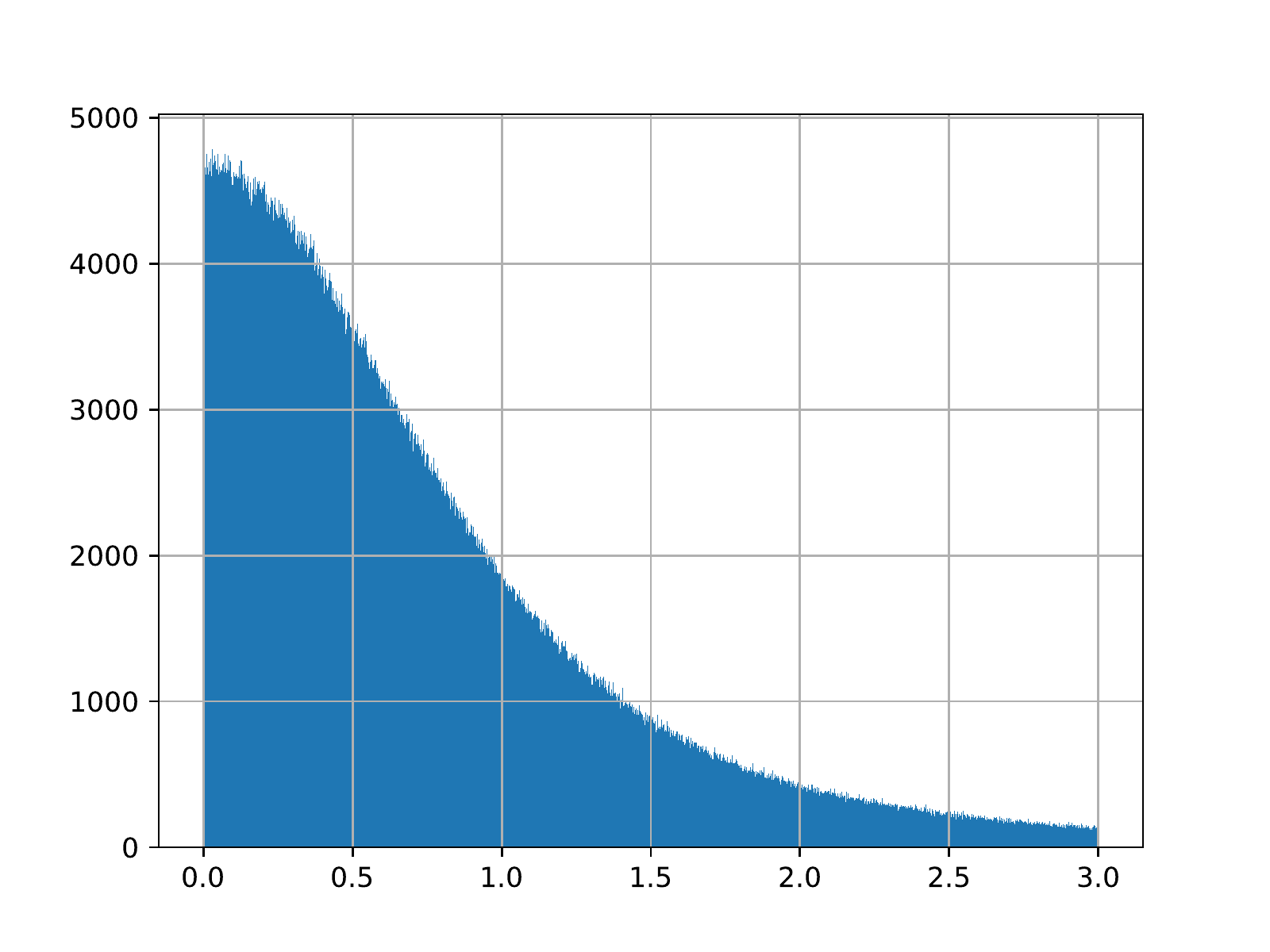}
    \caption{Distribution of $|A(2\times 10^4)|\cdot (\log(2\times 10^4))^{1/4} $ using $4\times 10^6$  samples of $\pm 1$ variables.}
    \label{fig:distribution_pm1}
\end{figure}

\subsection{Analysis of \cref{conj:other} for $\pm1$ variables} As in \cref{subsec:analysis_complex} and \ref{subsec:analysis_real}, we approximate $\E[|A(N)|]$ with the sample mean $ [|A(N)|]_{5 \times 10^7}$ for $1 \leq N \leq 2 \times 10^4$. From Figure \ref{fig:constant_pm1} and Table \ref{table:constant_pm1}, we observe that $[|A(N)|]_{5 \times 10^7}  \cdot (\log{N})^{1/4}$  is approaching a constant. Therefore, our data supports \cref{conj:other} for $\pm 1$ variables and   we estimate the  corresponding asymptotic constant to be $C \approx 0.896$ based on Table \ref{table:constant_pm1}.

\begin{figure}[h]
\centering

 \includegraphics[scale =0.6]{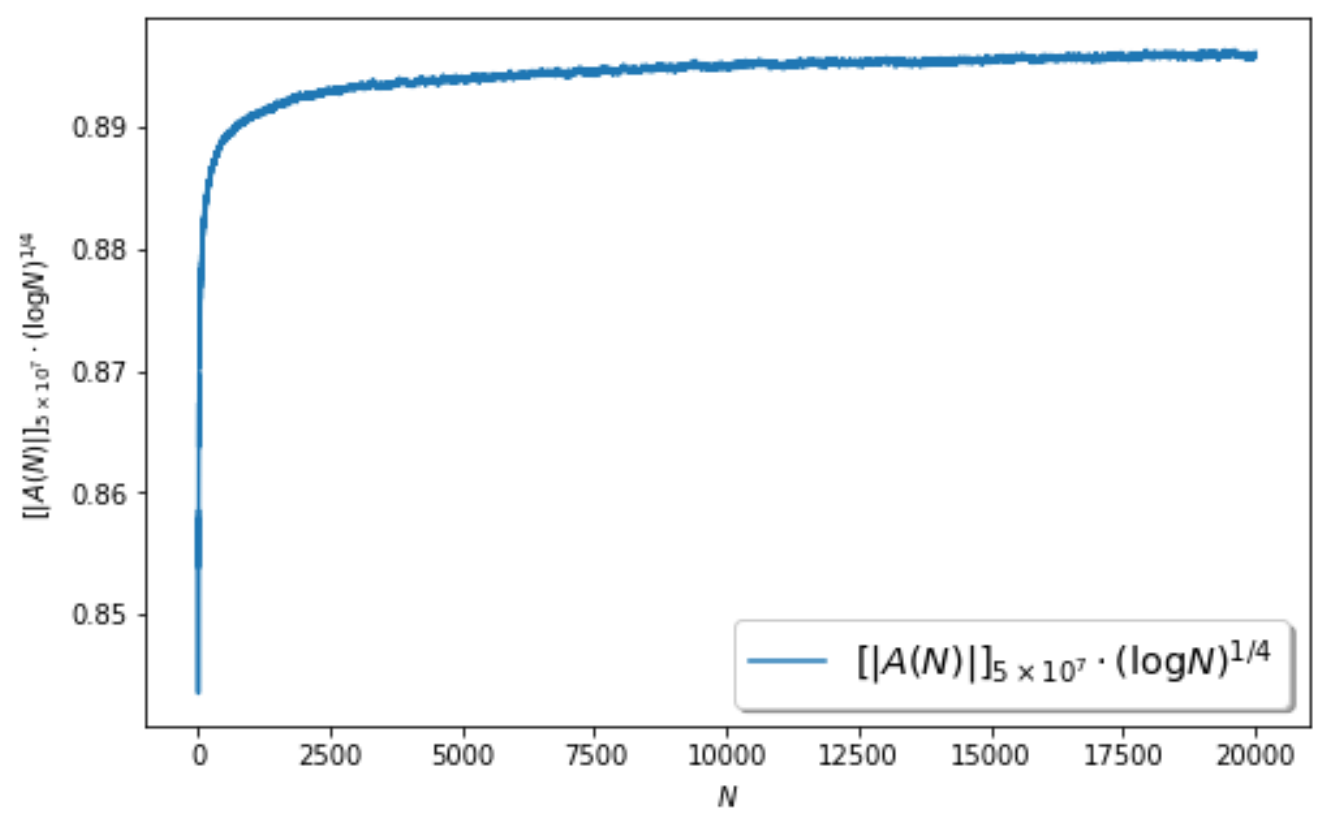}
    \caption{$[|A(N)|]_{5 \times 10^7} \cdot (\log{N})^{1/4}$ for $10 \leq N \leq 2\times10^4$.}
    \label{fig:constant_pm1}

\end{figure}

\begin{table}[h]
\centering
 \begin{tabular}{| c | c || c | c |} 
 \hline 
  $N$ & $[|A(N)|]_{5 \times 10^7}  \cdot (\log{N})^{1/4}$  & $N$ & $[|A(N)|]_{5 \times 10^7}  \cdot (\log{N})^{1/4}$ \\
 \hline 
 $1 \times 10^3$ & 0.8906 & $11 \times 10^3$ & 0.8949 \\
 \hline
 $2 \times 10^3$ & 0.8924 & $12 \times 10^3$ & 0.8952 \\
 \hline
$3 \times 10^3$ & 0.8931 & $13 \times 10^3$ & 0.8954\\
  \hline
 $4 \times 10^3$ & 0.8936 & $14 \times 10^3$ & 0.8952 \\
 \hline
 $5 \times 10^3$ & 0.8939 & $15 \times 10^3$ & 0.8954 \\
 \hline 
  $6 \times 10^3$ & 0.8943 & $16 \times 10^3$ & 0.8958  \\
 \hline
 $7 \times 10^3$ & 0.8944 & $17 \times 10^3$ & 0.8958 \\
 \hline
 $8 \times 10^3$ & 0.8946 & $18 \times 10^3$ & 0.8957 \\
  \hline
 $9 \times 10^3$ & 0.8946 & $19 \times 10^3$ & 0.8960 \\
 \hline
 $10 \times 10^3$ & 0.8947 & $20 \times 10^3$ & 0.8960 \\
 \hline
\end{tabular}
\caption{Values of $[|A(N)|]_{5 \times 10^7}  \cdot (\log{N})^{1/4}$ for $N \in \{i \times 10^3: 1 \leq i \leq 20\}.$ }
\label{table:constant_pm1}
\end{table}

\subsection{Robustness and sample size}
Again, we may reasonably infer that $[|A(n)|]_{5 \times 10^7} $ estimates $\mathbb{E}[\abs{A(n)}]$ accurately up to 3 decimal places with a high degree of certainty for all $1 \leq n \leq 2 \times 10^4$. We demonstrate this feature with two tests.

First, by fixing $N = 2\times 10^4$, we can see how the deviation $[|A(N)|]_{S}-[|A(N)|]_{4\times 10^6}$ evolves in Figure \ref{fig:smp_pm1} as $S$ grows. We conclude that \eqref{eq:bounded_deviation} holds as before, which supports our claim that the data presented in Table \ref{table:constant_pm1} is likely significant in 3 decimal places.

\begin{figure}[h]
\centering
\includegraphics[scale =0.55]{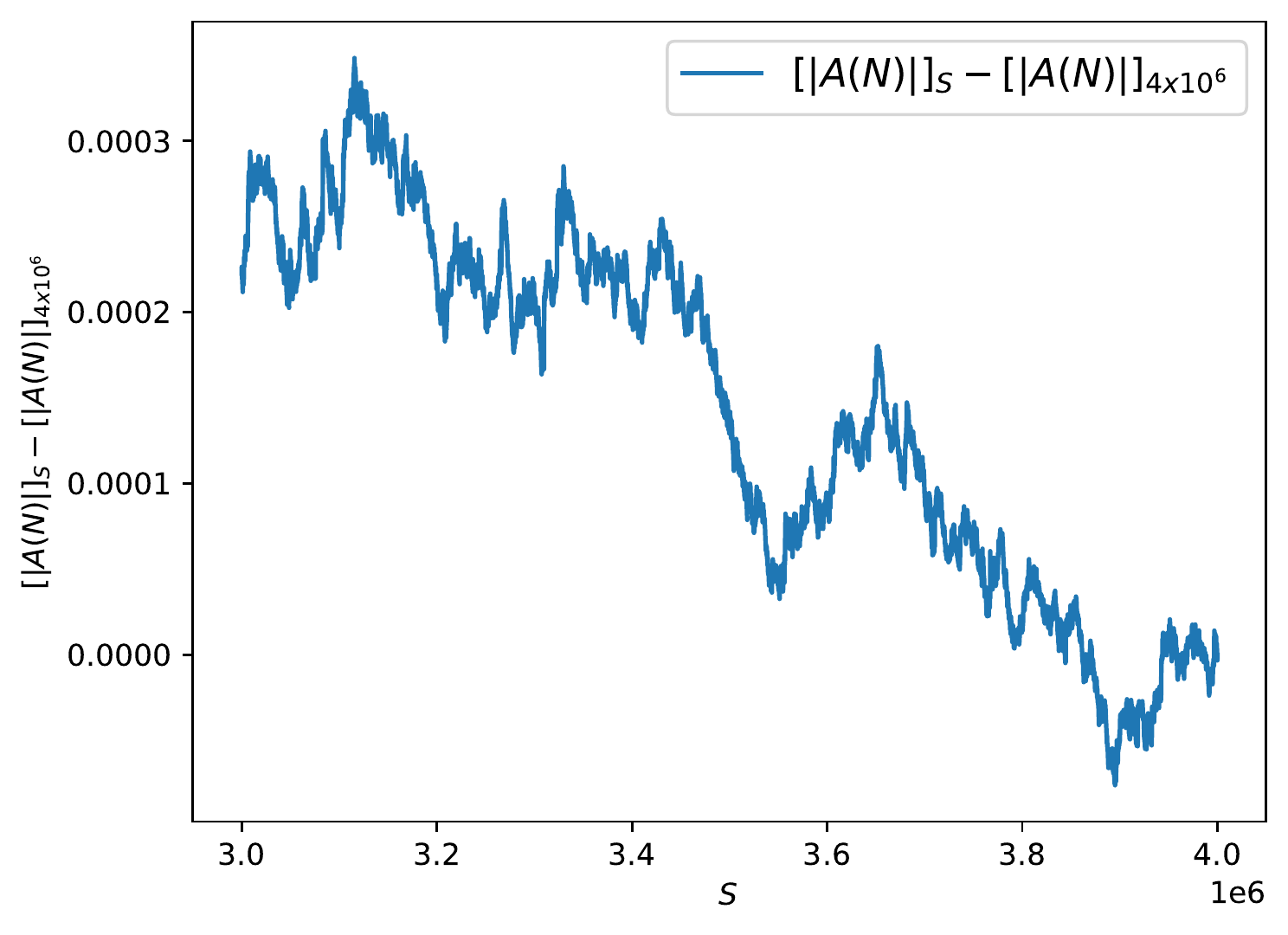}
\caption{ $[|A(N)|]_{S}-[|A(N)|]_{4\times 10^6}$ for $N = 2\times 10^4$, and $3\times 10^6 <  S\leq 4 \times 10^6$.}
\label{fig:smp_pm1}
\end{figure}

Second, we again study the deviation across six different data sets of sample means produced with $10^7$ samples. For $1 \leq i \leq 6 $ and $1 \leq N \leq 2 \times 10^4$, let $[|A(N)|]_{10^7}^{(i)}$ denote the sample mean of $|A(N)|$ from $i^{th}$ data set. Then $\delta_i(N)$, defined as usual by \eqref{absolute_deviation}, is shown in Figure \ref{fig:deviation_pm1} and its statistics are listed in Table \ref{table:deviation_pm1}. Notice that, for every $i$, $\delta_i(N) < 10^{-3}$ for almost all $N$. This supports the claim that $[|A(N)|]_{10^7}$ estimates  $\mathbb{E}[|A(N)|]$ accurately up to three decimal digits  with a high degree of certainty. Consequently, it is reasonable to infer that the data presented in Figure \ref{fig:constant_pm1} and Table \ref{table:constant_pm1} is significant up to three decimal digits. 

\begin{figure}[h]
\centering
\includegraphics[scale =0.6]{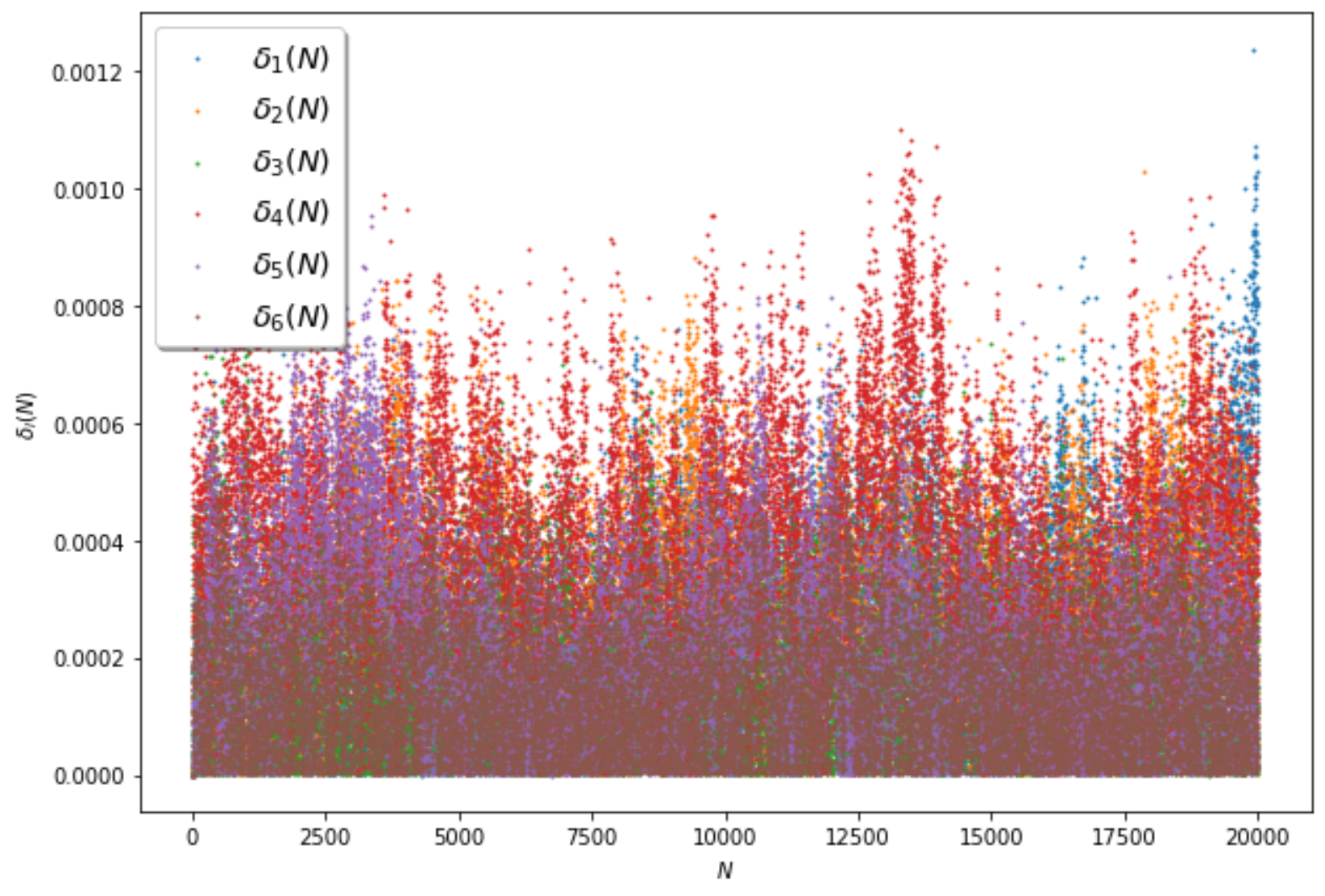}
    \caption{$\delta_i(N)$ for $1 \leq i \leq 6$ and $1 \leq N \leq 2 \times 10^4$.}
    \label{fig:deviation_pm1}

\end{figure}

 \begin{table}[h]
\begin{tabular}{|c | c | c | c |} 
\hline 
\\[-1em]

$i$ & $\text{Average of}\,\{\delta_i(N)\}_{N=1}^{2 \times 10^4}$& $\text{Maximum of}\,\{\delta_i(N)\}_{N=1}^{2 \times 10^4}$ & $\%$ of $N'$s with  $\delta_i(N) \geq 10^{-3}$ \\ 
 \hline \hline

1 & 0.0002& 0.0012 & 0.045 \\
\hline
2 & 0.0003 & 0.0010 & 0.005 \\
\hline
3 & 0.0002 & 0.0009 & 0.0 \\
\hline
4 & 0.0004 & 0.0011 & 0.085 \\
\hline
5 & 0.0002 & 0.00095 & 0.0\\
\hline
6 & 0.0002 & 0.0008 & 0.0 \\
\hline

 \hline

\end{tabular}
~\\
\caption{Descriptive statistics for $\delta_i(N)$ for each $1 \leq i \leq 6$.}
\label{table:deviation_pm1}
\end{table}

\section{Determining a product structure for the asymptotic constants} \label{Euler_product_conjecture}

Given that several number theoretic constants have Euler product expansions, it is natural to consider whether such an expansion might exist for the constant $C$ of Conjectures \ref{conj:complex} and \ref{conj:other}. This is the content of Questions \ref{q:complex}  and \ref{q:other}. The purpose of this section is to investigate these questions. 

It seems plausible that the answers to \cref{q:complex} and \cref{q:other} should match and so, due to their simplicity and amenability to conditional expectations, we will only investigate the case of $\pm 1$ variables in \cref{q:other}. 

\subsection{Setup} If \cref{q:other} has a positive answer, then a version of \cref{conj:other} (for $\pm 1$ variables) with conditional expectations should presumably hold. We formulate this stronger conjecture with some additional notation.

\begin{conjecture} \label{conj:conditional}
Let $(X(k))_{k \geq 1}$ be a sequence of independent random variables uniform on $\{\pm 1\}$ and define $(A(N))_{N \geq 0}$ by \eqref{power_series}.  For any finite subset $\K \subseteq \mathbb{N}$ of positive integers and any function $\epsilon: \K\to \{\pm 1\}$, there exists an absolute positive constant $C(\epsilon)$ such that
\[
\mathbb{E}\left[ \abs{A(N)} :  X({k}) = \epsilon(k) \text{ for all } k\in \K \right] \sim \frac{C(\epsilon)}{(\log{N})^{1/4}} \qquad  \text{as } N \to \infty. 
\]
\end{conjecture}
Note the unconditioned constant $C$ in \cref{conj:other} corresponds to the situation where $\K$ is the empty set. Now, by an analysis similar to the previous sections, computational evidence should support \cref{conj:conditional}, but the precise dependence of the constants $C(\epsilon)$ on the functions $\epsilon : \K \to \{\pm 1\}$ is not clear.  \cref{q:other} extends to these constants in a more precise manner.

\begin{question}\label{q:conditional}
Assume that \cref{conj:conditional}  is true. Do there exist sequences of real numbers $(\beta_k)_{k\geq 1}$,  $(\beta_k^+)_{k\geq 1}$ and  $(\beta_k^-)_{k\geq 1}$ such that for any finite subset $\K \subseteq \mathbb{N}$ and any $\epsilon: \K\to \{\pm 1\}$, 
\begin{equation}
    \label{eqn:conditional}
C(\epsilon) = \prod_{k \in \K} \beta_k^{\epsilon(k)} \cdot \prod_{k \not\in \K} \beta_k \, ?
\end{equation}
\end{question}
Again, notice that \eqref{eqn:product} corresponds to the situation where $\K$ is the empty set.  It seems plausible that the answers to Questions \ref{q:complex}, \ref{q:other}, and \ref{q:conditional} should all match (provided the corresponding conjectures are also all true). We shall therefore investigate the full strength of \cref{q:conditional} but, as we shall see, our computational evidence does not exhibit the multiplicative properties required by \eqref{eqn:conditional}. Thus, we hypothesize that the answers to Questions \ref{q:complex}, \ref{q:other}, and \ref{q:conditional} are all negative. 

Before we investigate \cref{q:conditional}, we record a lemma that illustrates  why conditioning on a single $X(k)$ for an  \textit{odd} integer $k$ does not yield new information. 

\begin{lemma}\label{lemma:odd_symmetry}
Let $k$ and $N$ be positive integers with $k$ odd and $\epsilon_k\in \{\pm 1\}$. Then 
\[
\mathbb{E}\big[ \abs{A(N)}: X(k) = \epsilon_k \big] = \mathbb{E}\big[ \abs{A(N)} \big]. 
\]
\end{lemma}
\begin{proof} Consider the idempotent map $\phi: \{ \pm 1\}^N \to \{\pm 1\}^N$ given by $\left( x_j \right)_{1\leq j \leq N} \mapsto \left((-1)^{j}\cdot x_j \right)_{1 \leq j \leq N}$. 
By \eqref{ceq:3}, it follows that for any partition $\lambda$ of $N$,
\[
a(\lambda; \phi(X)) = \prod_j (-1)^{j m_j(\lambda)} a(\lambda; X) =  (-1)^{N} a(\lambda; X),
\]
as $\sum_j jm_j(\lambda) = |\lambda| = N$. This implies by \eqref{eqn:partition_identity} that
\[
A(N;\phi(X)) = (-1)^N A(N; X). 
\]
Let $\Omega_+ \subseteq \{\pm 1\}^N$ (resp. $\Omega_-$) denote the set of tuples $(x_1,\dots,x_N) \in \{\pm 1\}^N$ such that $x_k = +1$ (resp. $x_k = -1$).  Since $k$ is odd, the function  $\phi$ maps $\Omega_+$ to a subset of $\Omega_-$. As $\phi$ is idempotent and $\Omega_{\pm}$ are disjoint sets whose union is $\{ \pm 1\}^N$, it follows that $\phi$ bijectively maps $\Omega_+$ to  $\Omega_-$ and vice versa.  Combining all of our observations, we have that
\[
\mathbb{E}\big[ \abs{A(N; X)}: X(k) = 1 \big] = \mathbb{E}\big[ \abs{A(N; \phi(X))}: X(k) = 1 \big]  =   \mathbb{E}\big[ \abs{A(N; X)}: X(k) = -1 \big].  
\]
This implies the desired result. 
\end{proof}

\begin{remark}
The argument extends to a similar result for any sequence $(X(k))_{k \geq 1}$ of independent random variables where $X$ and $-X$ are identically distributed.  
\end{remark}

\subsection{Computations for \cref{q:conditional}} In view of \cref{lemma:odd_symmetry}, we only investigate  \cref{q:conditional} for a subset $\K$ of even integers, namely $\K = \{2,4,6,8\}$. First, we approximate $C(\epsilon)$ for all functions $\epsilon: \K \to \{\pm 1\}$ by approximating $\mathbb{E}[\abs{A(N)}]$ at $N = 5000$ using $5 \times 10^6$ samples each\footnote{We used  $N = 5000$ instead of  $N = 20,000$ as in previous sections due to the limitations of our computational resources but we suspect that this is not a serious concern for the evaluation of \cref{q:conditional}.}. Heuristically, if the answer to \cref{q:conditional} were positive, the values $\beta_k$ for smaller $k$ would have greater contribution to the constant $C$ because the proportion of partitions with large parts is smaller than those with smaller parts. Table \ref{table:conditioned_C} presents the results of our estimation of $C(\epsilon)$ for $\K = \{2,4,6,8\}$, which confirms this heuristically.

\begin{table}
\centering
\begin{tabular}{|l|l||l|l|} 
\hline
$(\epsilon_2,\epsilon_4,\epsilon_6,\epsilon_8)$ &     $C(\epsilon)$ & $(\epsilon_2,\epsilon_4,\epsilon_6,\epsilon_8)$ &     $C(\epsilon)$   \\ 
\hline\hline
$(+1, +1, +1, +1)$        & 1.5340 & $(-1,+1,+1, +1)$       & 0.7578  \\ 
\hline
$(+1, +1, +1, -1)$       & 1.1399 & $(-1,+1,+1, -1)$      & 0.7707   \\ 
\hline
$(+1, +1, -1, +1)$       & 1.0018 & $(-1,+1, -1, +1)$      & 0.8066   \\ 
\hline
$(+1,+1, -1, -1)$      & 0.9819 & $(-1,+1, -1, -1)$     & 0.7274    \\ 
\hline
$(+1, -1,+1, +1)$       & 0.9554 & $(-1, -1,+1, +1)$      & 0.7819  \\ 
\hline
$(+1, -1,+1, -1)$      & 0.8184  & $(-1, -1,+1, -1)$     & 0.8003  \\ 
\hline
$(+1, -1, -1, +1)$      & 0.9008 & $(-1, -1, -1, +1)$     & 0.7145   \\ 
\hline
$(+1, -1, -1, -1)$     & 0.9404 & $(-1, -1, -1, -1)$    & 0.6738  \\ 
\hline
\end{tabular}
~\\[5pt]
\caption{Estimation of $C(\epsilon)$ with $\epsilon(k) = \epsilon_k$ for $k \in \K = \{2,4,6,8\}$} 
\label{table:conditioned_C}
\end{table}

Next, we compare various constants of conditional expectations.  
For any integer $k \in \mathbb{N}$, define $C_k^+$ (resp. $C_k^-$) to be the constant in \cref{conj:conditional} corresponding to the set $\K = \{k\}$ and the choice of function $\epsilon(k) = +1$ (resp. $\epsilon(k) = -1$). If the answer to \cref{q:conditional} is positive, then  
\[
C_k^+ = \beta_k^+ \prod_{j \neq k} \beta_j \quad \text{and} \quad C_k^- = \beta_k^- \prod_{j \neq k} \beta_j.
\]
For any finite subset $\K \subseteq \mathbb{N}$ and any function $\epsilon: \K\to \{\pm 1\}$, define the ratio 
\[
\rho(\epsilon)   \coloneqq \frac{C(\epsilon) \cdot C^{|\K|-1}}{\prod_{k\in \K}C_k^{\epsilon(k)}},\]
where $C$ is, as before, the conjectured unconditioned asymptotic constant. 

 \begin{table}[h]
	\begin{center}
		\begin{tabular}{|c||c|c|c|c|c|c|}
		\hline
			\diagbox{$(\epsilon_1,\epsilon_2)$}{$\K$} & $\{2,4\}$ & $\{2,6\}$ & $\{2,8\}$ & $\{4,6\}$ & $\{4,8\}$ & $\{6,8\}$ \\
			\hline \hline
			$(+1, +1)$ & 1.0433 & 1.0176 & 1.0191 & 1.0303 & 1.0195 & 1.0232 \\
			\hline
			$(+1, -1)$ & 0.9492 & 0.9803 & 0.9793 & 0.9661 & 0.9788 & 0.9748 \\\hline
			$(-1, +1)$ & 0.9406 & 0.9759 & 0.9738 & 0.9645 & 0.9772 & 0.9740 \\\hline
			$(-1, -1)$ & 1.0696 & 1.0270 & 1.0284 & 1.0397 & 1.0248 & 1.0283 \\\hline
		\end{tabular}
	\end{center}
	\caption{$\rho(\epsilon)$ for $\K = \{k_1,k_2\}$ with even $2\leq k_1 < k_2 \leq 8$ and $\epsilon(k) = \epsilon_{k}$.} 
	\label{table:conditioned_pair_ratios}
\end{table}

\begin{table}[h]
	\begin{center}
		\begin{tabular}{|c||c|c|c|c|}
			\hline
			\diagbox{$(\epsilon_1,\epsilon_2,\epsilon_3)$}{$\K$}& $\{2,4,6\}$ & $\{2,4,8\}$ & $\{2,6,8\}$ & $\{4,6,8\}$ \\
			\hline\hline
			$(+1, +1, +1)$ & 1.1336 & 1.0903 & 1.0932 & 1.0785 \\
			\hline
			$(+1, +1, -1)$ & 0.9421 & 0.9922 & 0.9353 & 0.9778 \\
			\hline
			$(+1, -1, +1)$ & 0.8816 & 0.9356 & 0.9360 & 0.9533 \\
			\hline
			$(+1, -1, -1)$ & 1.0251 & 0.9641 & 1.0285 & 0.9800 \\
			\hline
			$(-1, +1, +1)$ & 0.8885 & 0.9223 & 0.9272 & 0.9584 \\
			\hline
			$(-1, +1, -1)$ & 0.9990 & 0.9605 & 1.0288 & 0.9712 \\
			\hline
			$(-1, -1, +1)$ & 1.0783 & 1.0342 & 1.0261 & 0.9982 \\
			\hline
			$(-1, -1, -1)$ & 1.0599 & 1.1080 & 1.0280 & 1.0849 \\
			\hline
		\end{tabular}
	\end{center}
	\caption{$\rho(\epsilon)$ for $\K = \{k_1,k_2,k_3\}$ with even $2\leq k_1 < k_2 < k_3 \leq 8$ and $\epsilon(k) = \epsilon_{k}$.}
	\label{table:conditioned_trip_ratios}
\end{table}

If the answer to \cref{q:conditional} were positive then  \eqref{eqn:conditional} would imply that
 \begin{equation}
 \label{eq:euler_test}
     \rho(\epsilon) = 1
 \end{equation}
 for any function $\epsilon : \K \to \{\pm 1\}$. We examine the truth of \eqref{eq:euler_test} with $2$-subsets $\K \subseteq \{2,4,6,8\}$ in Table \ref{table:conditioned_pair_ratios} and with $3$-subsets in Table \ref{table:conditioned_trip_ratios}.
 They both   suggest that \eqref{eq:euler_test} does not hold, which indicates that the answer to Question \ref{q:conditional} (and hence Questions \ref{q:complex} and \ref{q:other}) should be negative.

\bibliographystyle{acm}
\bibliography{references.bib}

\begin{thebibliography}{10}

\bibitem{BaileyKeating-2021}
{\sc Bailey, E.~C., and Keating, J.~P.}
\newblock Maxima of log-correlated fields: some recent developments.
\newblock {\em arXiv preprint arXiv:2106.15141\/} (2021).

\bibitem{BarralKupiainenNikulaSaksmanWebb2015f}
{\sc Barral, J., Kupiainen, A., Nikula, M., Saksman, E., and Webb, C.}
\newblock Basic properties of critical lognormal multiplicative chaos.
\newblock {\em The Annals of Probability 43}, 5 (2015), 2205--2249.

\bibitem{ChhaibiNajnudel-2019}
{\sc Chhaibi, R., and Najnudel, J.}
\newblock On the circle ${G}{M}{C}^{\gamma} = \varprojlim {C}\beta {E}_n$ for
  $\gamma = \sqrt{\frac{2}{\beta}}$, ($\gamma \leq 1$).
\newblock {\em arXiv preprint arXiv:1904.00578\/} (2019).

\bibitem{DiaconisGamburd-2004}
{\sc Diaconis, P., and Gamburd, A.}
\newblock Random matrices, magic squares and matching polynomials.
\newblock {\em the electronic journal of combinatorics\/} (2004), R2--R2.

\bibitem{DuplantierRhodesSheffieldVargas2014u}
{\sc Duplantier, B., Rhodes, R., Sheffield, S., and Vargas, V.}
\newblock Renormalization of critical gaussian multiplicative chaos and {KPZ}
  relation.
\newblock {\em Communications in Mathematical Physics 330}, 1 (apr 2014),
  283--330.

\bibitem{DuplantierRhodesSheffieldVargas-2014}
{\sc Duplantier, B., Rhodes, R., Sheffield, S., and Vargas, V.}
\newblock Renormalization of critical {G}aussian multiplicative chaos and {KPZ}
  relation.
\newblock {\em Comm. Math. Phys. 330}, 1 (2014), 283--330.

\bibitem{DuplantierRhodesSheffieldVargas-2017}
{\sc Duplantier, B., Rhodes, R., Sheffield, S., and Vargas, V.}
\newblock Log-correlated gaussian fields: an overview.
\newblock {\em Geometry, analysis and probability\/} (2017), 191--216.

\bibitem{FyodorovHiaryKeating-2012}
{\sc Fyodorov, Y.~V., Hiary, G.~A., and Keating, J.~P.}
\newblock Freezing transition, characteristic polynomials of random matrices,
  and the {R}iemann zeta function.
\newblock {\em Physical review letters 108}, 17 (2012), 170601.

\bibitem{FyodorovKeating-2014}
{\sc Fyodorov, Y.~V., and Keating, J.~P.}
\newblock Freezing transitions and extreme values: random matrix theory, and
  disordered landscapes.
\newblock {\em Philosophical Transactions of the Royal Society A: Mathematical,
  Physical and Engineering Sciences 372}, 2007 (2014), 20120503.

\bibitem{gorodetskyMagicSquaresSymmetric2021}
{\sc Gorodetsky, O.}
\newblock Magic squares and the symmetric group.
\newblock {\em arXiv:2102.11966 [math]\/} (Feb. 2021).

\bibitem{harper_moments_2020}
{\sc Harper, A.~J.}
\newblock Moments of random multiplicative functions, {I}: Low moments, better
  than squareroot cancellation, and critical multiplicative chaos.
\newblock {\em Forum of Mathematics, Pi 8\/} (2020), e1.

\bibitem{HughesKeatingOConnell-2000}
{\sc Hughes, C.~P., Keating, J.~P., and O'Connell, N.}
\newblock Random matrix theory and the derivative of the {R}iemann zeta
  function.
\newblock {\em R. Soc. Lond. Proc. Ser. A Math. Phys. Eng. Sci. 456}, 2003
  (2000), 2611--2627.

\bibitem{NajnudelPaquetteSimm-2020}
{\sc Najnudel, J., Paquette, E., and Simm, N.}
\newblock Secular coefficients and the holomorphic multiplicative chaos.
\newblock {\em arXiv preprint arXiv:2011.01823\/} (2020).

\bibitem{RhodesVargas-2014}
{\sc Rhodes, R., and Vargas, V.}
\newblock Gaussian multiplicative chaos and applications: a review.
\newblock {\em Probab. Surv. 11\/} (2014), 315--392.

\bibitem{SoundararajanZaman-2021}
{\sc Soundararajan, K., and Zaman, A.}
\newblock A model problem for multiplicative chaos in number theory.
\newblock {\em arXiv preprint arXiv:2108.07264\/} (2021).

\end{thebibliography}
\parindent0pt
\end{document}